\documentclass[11pt,a4paper]{article}
\usepackage{amsmath,amssymb,amsthm}

\newcommand{\bim}[3]{\mathord{\raisebox{-0.4ex}[0ex][0ex]{\scriptsize $#1$}{#2}\hspace{-0.2ex}\raisebox{-0.4ex}[0ex][0ex]{\scriptsize $#3$}}}

\newcommand{\lmo}[2]{\mathord{\raisebox{-0.4ex}[0ex][0ex]{\scriptsize $#1$}{#2}}}

\newcommand{\rZ}{\operatorname{Z}}
\newcommand{\rL}{\operatorname{L}}
\newcommand{\Out}{\operatorname{Out}}
\newcommand{\Ad}{\operatorname{Ad}}

\newcommand{\Aut}{\operatorname{Aut}}
\newcommand{\FAlg}{\operatorname{FAlg}}
\newcommand{\SL}{\operatorname{SL}}
\newcommand{\QN}{\operatorname{QN}}
\newcommand{\B}{\operatorname{B}}
\newcommand{\Tr}{\operatorname{Tr}}

\newcommand{\M}{\operatorname{M}}
\newcommand{\Rep}{\operatorname{Rep}}

\newcommand{\morph}{\operatorname{Mor}}
\newcommand{\Mor}{\operatorname{Mor}}

\newcommand{\irr}{\operatorname{Irr}}

\newcommand{\lspan}{\operatorname{span}}

\newcommand{\Comm}{\operatorname{Comm}}
\newcommand{\eL}{\operatorname{L}}
\newcommand{\Bimod}{\operatorname{Bimod}}

\newcommand{\cT}{\mathcal{T}}

\newcommand{\cM}{\mathcal M}
\newcommand{\cF}{\mathcal{F}}
\newcommand{\cL}{\mathcal{L}}
\newcommand{\cA}{\mathcal{A}}
\newcommand{\cG}{\mathcal{G}}
\newcommand{\cH}{\mathcal{H}}
\newcommand{\cU}{\mathcal{U}}

\newcommand{\cK}{\mathcal{K}}
\newcommand{\cZ}{\mathcal{Z}}
\newcommand{\cC}{\mathcal{C}}

\newcommand{\N}{\mathbb{N}}
\newcommand{\Z}{\mathbb{Z}}
\newcommand{\R}{\mathbb{R}}
\newcommand{\C}{\mathbb{C}}
\newcommand{\Q}{\mathbb{Q}}

\newcommand{\id}{\mathord{\text{\rm id}}}

\newcommand{\si}{\sigma}
\newcommand{\be}{\beta}
\newcommand{\al}{\alpha}
\newcommand{\om}{\omega}

\newcommand{\vphi}{\varphi}

\newcommand{\eps}{\epsilon}

\newcommand{\recht}{\rightarrow}

\newcommand{\actson}{\curvearrowright}

\newcommand{\droite}{\rightarrow}

\newcommand{\ot}{\otimes}
\newcommand{\otvN}{\overline{\otimes}}

\newcommand{\ox}{\overline{x}}

\newtheorem{theo}{Theorem}[section]
\newtheorem{prop}[theo]{Proposition}
\newtheorem{lemm}[theo]{Lemma}

{\theoremstyle{definition}
\newtheorem{defi}[theo]{Definition}
\newtheorem{rema}[theo]{Remark}
\newtheorem{exam}[theo]{Example}
}

\begin{document}

\begin{center}
{\LARGE \bf {The representation category of any compact group \vspace{0.5ex} is the bimodule category of a II$_\mathbf{1}$ factor}}

\bigskip

{\sc by S\'{e}bastien Falgui\`{e}res$^{(1,2)}$ and Stefaan Vaes$^{(1,2)}$\setcounter{footnote}{1}\footnotetext{Partially
    supported by ERC Starting Grant VNALG-200749 and Research
    Programme G.0231.07 of the Research Foundation --
    Flanders (FWO)}\setcounter{footnote}{2}\footnotetext{Department of Mathematics;
    K.U.Leuven; Celestijnenlaan 200B; B--3001 Leuven (Belgium).
    \\ E-mail: sebastien.falguieres@wis.kuleuven.be, stefaan.vaes@wis.kuleuven.be}}
\end{center}

\begin{abstract}
\noindent We prove that given any compact group $G$, there exists a minimal action of $G$ on a II$_1$ factor $M$ such that the bimodule category of the fixed-point II$_1$ factor $M^G$ is naturally equivalent with the representation category of $G$. In particular, all subfactors of $M^G$ with finite Jones index can be described explicitly.
\end{abstract}

\section{Introduction and statements of main results}

One of the richest invariants of a II$_1$ factor $P$ is the bimodule category $\Bimod(P)$ consisting of all $P$-$P$-bimodules $\bim{P}{\cH}{P}$ (see \cite{connes NCG,Popa correspondences}) of finite Jones index: $\dim(\lmo{P}{\cH}) < \infty$ and $\dim(\cH_P) < \infty$. Equipped with the Connes tensor product, $\Bimod(P)$ is a C$^*$-tensor category. Note that $\Bimod(P)$ contains both the fundamental group $\cF(P)$ and the outer automorphism group $\Out(P)$, because every $*$-isomorphism $\pi : P \recht pPp$ yields the bimodule $\bim{P}{\eL^2(P)p}{\pi(P)}$. More precisely, the group-like elements in $\Bimod(P)$ form an extension of $\cF(P)$ by $\Out(P)$. Moreover, $\Bimod(P)$ encodes, in a certain sense, all subfactors $P_0 \subset P$ of finite Jones index \cite{jones}: performing Jones' basic construction, we get $P_0 \subset P \subset P_1$ and obtain the $P$-$P$-bimodule $\bim{P}{\eL^2(P_1)}{P}$. As a result, it seemed until recently quite hopeless to explicitly compute $\Bimod(P)$ for any II$_1$ factor $P$.

But, in \cite{P1,P2,P3}, Sorin Popa obtained several breakthrough rigidity results for II$_1$ factors, which allowed in particular to compute invariants like $\cF(P)$ and $\Out(P)$ for concrete II$_1$ factors $P$. Without being exhaustive, we mention the following results: in \cite{P3}, Popa obtained the first II$_1$ factors having trivial fundamental group, while in \cite{P1}, he constructed examples with prescribed countable fundamental group. Very recently, Popa and the second author \cite{PV2,PV3} proved that the invariant $\cF(P)$ actually ranges over a large family of uncountable subgroups of $\R_+$. In \cite{IPP}, Ioana, Peterson and Popa proved the existence of II$_1$ factors $P$ such that $\Out(P)$ is any prescribed second countable compact abelian group. In particular, this settled the long standing open problem of the possible existence of II$_1$ factors only having inner automorphisms. The first concrete computations of $\Out(P)$ were given by Popa and the second author in \cite{PV1} and later refined in \cite{V2}. On the other hand, we proved in \cite{FV} that also all non-abelian compact groups arise as $\Out(P)$.

The II$_1$ factors studied in \cite{IPP} are amalgamated free products $M = M_0 *_N M_1$ (see Section \ref{sec.amal} for definitions). The main result of \cite{IPP} says that a von Neumann subalgebra $Q$ of $M$ having the property (T) of Connes and Jones (see \cite{CJ} and Section \ref{sec.T} below), or just having the relative property (T) in the sense of Popa \cite{P3}, must essentially be contained in either $M_0$ or $M_1$. In particular, if $M_0$ has itself property (T) and $M_1$ has, say, the Haagerup property, every automorphism of $M$ must preserve $M_0$ globally. This is the starting point to compute $\Out(M)$ in certain particular cases, leading to the above mentioned results of \cite{FV,IPP}.

In \cite{V1}, the scope of the methods of \cite{IPP} was enlarged so that in certain cases not only $\Out(M)$ but also $\Bimod(M)$ could actually be computed. The main result of \cite{V1} proves the existence of II$_1$ factors $M$ having trivial bimodule category and hence also trivial subfactor structure, trivial fundamental group and trivial outer automorphism group. Note however that the results in \cite{FV,IPP,V1} are existence theorems. The first concrete II$_1$ factors with trivial bimodule category were given in \cite{V2}, which included as well concrete examples of II$_1$ factors where $\Bimod(P)$ is a Hecke-like category.

We prove in this paper that the representation category of an arbitrary compact group $G$ can be realized as the bimodule category of a II$_1$ factor. More precisely, we prove the existence of a minimal action $G \actson M$ of $G$ on a II$_1$ factor $M$, such that the bimodule category $\Bimod(M^G)$ of the fixed point algebra $M^G$ can be identified with the representation category $\Rep(G)$. Note that $G \actson M$ is called minimal if $G \recht \Aut(M)$ is injective and the fixed point algebra $M^G$ has trivial relative commutant in $M$, i.e.\ $M \cap (M^G)' = \C 1$. Whenever $G \actson M$ is a minimal action, there is a natural embedding of $\Rep(G)$ into $\Bimod(M^G)$ (see \cite{roberts} and Section \ref{subsec.minimal}). The striking point is that there exist minimal actions such that this embedding is surjective (up to unitary equivalence). As in \cite{FV,IPP,V1}, our result is an existence theorem, involving a Baire category argument (Theorem \ref{Gdelta}).

As in \cite{FV,IPP}, the II$_1$ factor $M$ in the previous paragraph is of the form $M = M_0 *_N M_1$ and the action $G \actson M$ is such that $G$ acts trivially on $M_0$, leaves $M_1$ globally invariant and satisfies $M_1^G = N$. Our main theorem is the following.

\begin{theo}\label{thm.main}
Let $G$ be a second countable compact group. There exists a II$_1$ factor $M$ and a minimal action $G \actson M$ such that, writing $P := M^G$, every finite index $P$-$P$-bimodule is isomorphic with $\bim{P}{\Mor(H_\pi,\eL^2(M))}{P}$ for a uniquely determined finite dimensional unitary representation $\pi : G \recht \cU(H_\pi)$.

More precisely, $\Rep(G) \recht \Bimod(P) : \pi \mapsto \bim{P}{\Mor(H_\pi,\eL^2(M))}{P}$ defines an equivalence of C$^*$-tensor categories.
\end{theo}

Theorem \ref{thm.main} provides examples of II$_1$ factors for which all finite index bimodules over $P = M^G$ can be listed explicitly, labeled by the finite dimensional unitary representations of $G$. Since the category of finite index bimodules over $P$ encodes in a certain way all finite index subfactors of $P$, these can be explicitly listed as well. In particular, Jones' invariant \cite{jones}
$$\cC(P) := \{[P : P_0] \mid P_0 \subset P \;\text{irreducible, finite index subfactor}\}$$
can be explicitly computed for the II$_1$ factors $P = M^G$ given by Theorem \ref{thm.main}. The precise result goes as follows. We make use of Jones' tunnel construction \cite[Corollary 3.1.9]{jones} saying that for every finite index inclusion of II$_1$ factors $P \subset N$, there exists a finite index subfactor $P_0 \subset P$ such that $P_0 \subset P \subset N$ is the basic construction. Moreover, $P_0$ is uniquely determined up to unitary conjugacy in $P$.

\begin{theo}\label{thm.subfactors}
Let $G \overset{\sigma}{\actson} M$ be a minimal action of the second countable compact group $G$ on the II$_1$ factor $M$ and write $P = M^G$. Assume that $\sigma$ satisfies the conclusion Theorem \ref{thm.main}, meaning that every finite index $P$-$P$-bimodule is of the form
$\bim{P}{\Mor(H_\pi,\eL^2(M))}{P}$ for some finite dimensional unitary representation $\pi$ of $G$.

Whenever $G \overset{\alpha}{\actson} A$ is an action on the finite dimensional von Neumann algebra $A$ with $\cZ(A)^G = \C 1$, define the finite index subfactor $P(\alpha) \subset P$ such that $1 \ot P(\alpha) \subset 1 \ot P \subset (A \ot M)^{\al \ot \si}$ is the basic construction. Here $(\al \ot \si)_g := \al_g \ot \si_g$ and we note that $P(\alpha)$ is uniquely defined up to unitary conjugacy in $P$.
\begin{itemize}
\item Every finite index subfactor of $P$ is unitarily conjugate to one of the $P(\alpha)$.
\item $[P : P(\alpha)] = \dim A$ and $P(\alpha) \subset P$ is irreducible iff $A^G = \C 1$.
\item If $G \overset{\alpha}{\actson} A$ and $G \overset{\beta}{\actson} B$ satisfy $\cZ(A)^G = \C1$ and $\cZ(B)^G = 1$, then the subfactors $P(\alpha)$ and $P(\beta)$ of $P$ are unitarily conjugate in $P$ iff there exists a $*$-isomorphism $\pi : A \recht B$ satisfying $\beta_g \circ \pi = \pi \circ \alpha_g$ for all $g \in G$.
\end{itemize}
In particular, the set of index values of irreducible finite index subfactors of $P$ is given by
$$\cC(P) = \{ \dim(A) \mid A \;\text{finite dimensional von Neumann algebra}, \; G \actson A \; , \; A^G = \C1 \} \; .$$
\end{theo}

\subsection*{Acknowledgment}

The authors would like to thank the referee for carefully reading the paper and making several suggestions that have lead to a considerable improvement of the exposition.

\section{Preliminaries}

\subsection{The $*$-algebra of operators affiliated with a II$_1$ factor} \label{subsec.affiliated}

Let $M \subset \B(H)$ be a II$_1$ factor with normal tracial state $\tau$. Denote by $\cM$ the closed densely defined operators affiliated with $M$. By
\cite[Theorem XV, page 229]{MvN1}, we know that $\cM$ is a $*$-algebra, where sum and product are defined as the closure of sum and product on the natural domains and where the adjoint is the usual adjoint of operators. Denote by $\cM^+$ the positive self-adjoint operators affiliated with $M$. Then, $\tau$ has a natural extension to a positive-linear map $\cM^+ \recht [0,+\infty]$. Define, for $x \in \cM$, $|x| := (x^*x)^{1/2}$, $\|x\|_2 := \tau(x^*x)^{1/2}$ and $\|x\|_1 := \tau(|x|)$. Put
$$\rL^2(M) := \{x \in \cM \mid \|x\|_2 < \infty \} \quad\text{and}\quad \rL^1(M) := \{x \in \cM \mid \|x\|_1 < \infty \} \; .$$
Actually, $\rL^1(M)$ is the linear span of $\{x \in \cM^+ \mid \tau(x) <\infty\}$ and $\tau$ extends to a linear map $\rL^1(M) \recht \C$. Both $\rL^2(M)$ and $\rL^1(M)$ are stable under the adjoint and are $M$-$M$-bimodules. Finally, the product of two elements in $\rL^2(M)$ belongs to $\rL^1(M)$, the Cauchy-Schwartz inequality holds and the scalar product $\langle x,y\rangle := \tau(x^*y)$ turns $\rL^2(M)$ into a Hilbert space.

Every $x \in \cM$ has a unique polar decomposition, $x = u |x|$, where $u$ is a partial isometry in $M$ with $u^*u$ equal to the support projection of $|x|$. If $N \subset M$ is an irreducible subfactor, meaning that $N' \cap M = \C 1$, every element $x \in \cM$ satisfying $a x = x a$ for all $a \in N$, belongs to $\C 1$ as well. Indeed, if $x = u |x|$ is the polar decomposition of $x$, the uniqueness of the polar decomposition implies that $u$ and $|x|$ commute with all unitaries in $N$. Hence, $u$ is scalar. By the uniqueness of the spectral decomposition of $|x|$, all spectral projections of $|x|$ are scalar and hence, also $|x|$ follows scalar.

\subsection{Some notational conventions}

For any von Neumann algebra $M$ we denote $M^n:= \M_n(\C) \ot M$. We write $\C^n (\C^m)^*$ instead of $\B(\C^m,\C^n)$. We implicitly consider $\C^n (\C^m)^*$ as a Hilbert space with scalar product $\langle \xi,\eta \rangle = \Tr(\xi^* \eta)$. Obviously, $\C^n (\C^m)^*$ is an $\M_n(\C)$-$\M_m(\C)$-bimodule. We denote by $e_i \in \C^n$ the natural vectors. All $*$-homomorphisms between von Neumann algebras are implicitly assumed to be normal.

\subsection{The bimodule category and fusion algebra of a II$_1$ factor} \label{subsec.bimodule}

Let $(M,\tau)$ be a von Neumann algebra with faithful normal tracial state and $\cH_M$ a right Hilbert $M$-module. There exists a projection $p \in \B(\ell^2(\N)) \otvN M$ such that $\cH_M \cong p \big ( \ell^2(\N) \otvN \eL^2(M) \big )_M$ and this projection $p$ is uniquely defined up to equivalence of projections in $\B(\ell^2(\N)) \otvN M$. We denote $\dim(\cH_M) := (\Tr \ot \tau)(p)$. Observe that the number $\dim(\cH_M)$ depends on the choice of tracial state $\tau$ in the non-factorial case. An $N$-$M$-bimodule $\bim{N}{\cH}{M}$ is said to be of \emph{finite Jones index} if $\dim( \lmo{N}{\cH}) < \infty$ and $\dim(\cH_M) < \infty$. In particular, the \emph{Jones index} of a subfactor $N \subset M$ is defined as $[M : N]:= \dim (\eL^2(M)_N)$, see \cite{jones}.

We constantly use the following well known principle: if $N \subset M$ is a finite index subfactor and $Q \subset N$ a von Neumann subalgebra such that $Q' \cap N$ is finite dimensional, then also $Q' \cap M$ is finite dimensional (see e.g.\ \cite[Lemma A.3]{V2}).

Let $M,N,P$ be von Neumann algebras with faithful normal tracial states and fix bimodules $\bim{M}{\cH}{N}$ and $\bim{N}{\cK}{P}$. We briefly recall the construction of the \emph{Connes tensor product} $\bim{M}{(\cH \ot_N \cK)}{P}$ and refer to \cite[V.Appendix B]{connes NCG} for details. Denote by $\cH^0$ the set of vectors $\xi \in \cH$ such that the linear map $N \recht \cH : a \mapsto \xi a$ extends to a bounded operator $L_\xi : \rL^2(N) \recht \cH$. Then, $\cH^0$ is a dense subspace of $\cH$. One defines an $N$-valued scalar product on $\cH^0$ by setting $\langle \xi,\eta \rangle_N := L_\xi^* L_\eta$. The Connes tensor product $\cH \ot_N \cK$ is defined as the separation and completion of the algebraic tensor product $\cH^0 \ot_{\text{\rm alg}} \cK$ for the scalar product
$$\langle a \ot \xi , b \ot \eta \rangle:= \langle \xi , \langle a, b\rangle_N \eta \rangle \; . $$
The Hilbert space $\cH \ot_N \cK$ is turned into an $M$-$P$-bimodule in the following way:
$$a \cdot (b \ot \xi) = ab \ot \xi \quad\text{and}\quad (b \ot \xi) \cdot a = b \ot(\xi a) \; .$$

Whenever $p \in \B(\ell^2(\N)) \otvN N$ is a projection and $\psi : P \recht p(\B(\ell^2(\N)) \otvN N)p$ is a $*$-homomorphism, define the $N$-$P$-bimodule $H(\psi)$ on the Hilbert space $(\ell^2(\N)^* \ot \rL^2(N))p$ with left and right module actions given by
$$a \cdot \xi := a \xi \quad\text{and}\quad \xi \cdot b = \xi \psi(b) \; .$$
Every $N$-$P$-bimodule is isomorphic with an $N$-$P$-bimodule of the form $H(\psi)$. Furthermore, if $\psi : P \recht p(\B(\ell^2(\N)) \otvN N)p$ and $\eta : P \recht q(\B(\ell^2(\N)) \otvN N)q$, then $\bim{N}{H(\psi)}{P} \cong \bim{N}{H(\eta)}{P}$ if and only if there exists $u \in \B(\ell^2(\N)) \otvN N$ satisfying $u u^* = p$, $u^* u = q$ and $\psi(a) = u \eta(a) u^*$ for all $a \in P$.

Whenever $\bim{M}{\cH}{N}$ is an $M$-$N$-bimodule, the Connes tensor product $\bim{M}{(\cH \ot_N H(\psi))}{P}$ is isomorphic with the $M$-$P$-bimodule defined on the Hilbert space $(\ell^2(\N)^* \ot \cH)p$ with left and right module actions given by $a \cdot \xi = a \xi$ and $\xi \cdot b = \xi \psi(b)$. In particular, $H(\rho) \ot_N H(\psi) \cong H((\id \ot \rho)\psi)$.

In the previous two paragraphs, one can analogously describe bimodules by homomorphisms on the left: $\bim{M}{\cH}{N} \cong \bim{\vphi(M)}{p(\ell^2(\N) \ot \rL^2(N))}{N}$ for a $*$-homomorphism $\vphi : M \recht p (\B(\ell^2(\N) \otvN N)) p$.

The \emph{contragredient} of an $M$-$N$-bimodule $\bim{M}{\cH}{N}$ is defined on the conjugate Hilbert space $\cH^*$ with bimodule actions given by $a \cdot \xi^* := (\xi a^*)^*$ and $\xi^* \cdot b := (b^* \xi)^*$.

From now on, fix a II$_1$ factor $M$. The category $\Bimod(M)$ consists of all finite index $M$-$M$-bimodules, with morphisms given by the $M$-$M$-bimodular maps. We refer to \cite{bisch} for background material and results on bimodules and fusion algebras, in particular in relation with subfactors. Every finite index $M$-$M$-bimodule is isomorphic with an $H(\psi)$ for some finite index inclusion $\psi : M \recht p M^n p$.

One also defines the \emph{fusion algebra} $\FAlg(M)$ of $M$ as the set of finite index $M$-$M$-bimodules modulo unitary equivalence.
We recall that an abstract fusion algebra $\cA$ is a free $\N$-module $\N[\cG]$ equipped with the following additional structure:
\begin{itemize}
\item an associative and distributive product operation, and a multiplicative unit element $e \in \cG$,
\item an additive, anti-multiplicative, involutive map $x \mapsto \overline{x}$, called \emph{conjugation},
\end{itemize}
satisfying Frobenius reciprocity: defining the numbers $m(x,y;z) \in \N$ for $x,y,z \in \cG$ through the formula
$$x y = \sum_z m(x,y;z) z\; ,$$
one has $m(x,y;z) = m(\ox,z;y) = m(z,\overline{y}; x)$ for all $x,y,z \in \cG$.

The base $\cG$ of the fusion algebra $\cA$ is canonically determined: these are exactly the non-zero elements of $\cA$ that cannot be expressed as the sum of two non-zero elements. The elements of $\cG$ are called the \emph{irreducible elements} of the fusion algebra $\cA$.

Two examples of fusion algebras arise as follows.
\begin{itemize}
\item Let $\Gamma$ be a group and define $\cA = \N[\Gamma]$.
\item Let $G$ be a compact group and define the fusion algebra $\Rep(G)$ as the set of equivalence classes of finite dimensional unitary representations of $G$. The operations on $\Rep(G)$ are given by direct sum and tensor product of representations.
\end{itemize}

We end with the following probably well known lemma. For convenience, we give a proof.

\begin{lemm}\label{lem.finite-dim}
Let $N \subset M$ be an irreducible inclusion of II$_1$ factors and $\bim{M}{\cK}{M}$ a finite index $M$-$M$-bimodule. Whenever $\bim{N}{\cH}{N}$ is a finite index $N$-$N$-bimodule, the vector space of bounded $N$-$N$-bimodular operators from $\cH$ to $\cK$, is finite dimensional.
\end{lemm}
\begin{proof}
Write $\bim{M}{\cK}{M} \cong \bim{\psi(M)}{p(\C^n \ot \rL^2(M))}{M}$ and $\bim{N}{\cH}{N} \cong \bim{\gamma(N)}{q(\C^m \ot \rL^2(N))}{N}$ for some finite index inclusions $\psi : M \recht p M^n p$ and $\gamma : N \recht qN^m q$. Define
$$\cL = \{T \in p(\C^n (\C^m)^* \ot M) q \mid \psi(a) T = T \gamma(a) \quad\text{for all}\;\; a \in N \} \; .$$
For every $T \in \cL$, left multiplication by $T$ defines an $N$-$N$-bimodular map $q(\C^m \ot \rL^2(N)) \recht p(\C^n \ot \rL^2(M))$. Conversely, let $\theta : q(\C^m \ot \rL^2(N)) \recht p(\C^n \ot \rL^2(M))$ be a bounded $N$-$N$-bimodular operator. Define $T \in
p(\C^n (\C^m)^* \ot \rL^2(M)) q$ by the formula
$$T = \sum_{i=1}^m \theta(q(e_i \ot 1)) (e_i^* \ot 1) \; .$$
It follows that $\psi(a) T = T \gamma(a)$ for all $a \in N$. Define $A := pM^n p \cap \psi(N)'$. Since $N \subset M$ is irreducible and $\psi(M) \subset pM^n p$ has finite index, $A$ is finite dimensional. The operator $TT^*$ belongs to $p (\M_n(\C) \ot \rL^1(M)) p$ and commutes with $\psi(N)$. Hence, $TT^*$ is affiliated with $A$ and in particular, bounded. Hence, $T \in \cL$.

So, we have to prove that $\cL$ is finite dimensional. Let $p_1,\ldots,p_r$ be a maximal set of mutually orthogonal minimal projections in $A$ for which there exist $v_i \in \cL$ with $v_i v_i^* = p_i$. For every $i$, let $q_{i1},\ldots,q_{is_i}$ be a maximal set of mutually orthogonal projections in $qM^m q$ for which there exist $v_{ij} \in \cL$ satisfying $v_{ij} v_{ij}^* = p_i$ and $v_{ij}^* v_{ij} = q_{ij}$. Since $\Tr(q_{ij}) = \Tr(p_i)$ for every $j$, it follows that $s_i < \infty$. It is now easy to check that
$$\cL = \lspan \{v_{ij} \mid i=1,\ldots,r, j = 1,\ldots,s_i\} \; .$$
\end{proof}

\subsection{Connes tensor product versus product in a given module}

For the convenience of the reader, we prove the following elementary and probably well known lemma. It will be used several times in this paper.

\begin{lemm}\label{lem.product}
Let $N \subset M$ be an irreducible inclusion of II$_1$ factors. Suppose that $\cK \subset \rL^2(M)$ is an $N$-$N$-subbimodule of finite index.

\begin{enumerate}
\item Choose a projection $p \in N^n$, a finite index inclusion $\vphi : N \recht pN^n p$ and an $N$-$N$-bimodular unitary
$$U : \bim{\vphi(N)}{p (\C^n \ot \rL^2(N))}{N} \recht \bim{N}{\cK}{N} \; .$$
Then, $U \bigl(p (\C^n \ot N)\bigr) = \cK \cap M$ and defining $v \in (\C^n)^* \ot M$ by the formula
$$v := \sum_{i=1}^n e_i^* \ot v_i \quad\text{with}\quad v_i := U(p(e_i \ot 1)) \; ,$$
we have $v p = v$, $a v = v \vphi(a)$ for all $a \in N$ and $U(\xi) = v \xi$ for all $\xi \in p (\C^n \ot \rL^2(N))$. In particular,
$\cK \cap M = \lspan\{ v_i N \mid i=1,\ldots,n\}$ and $\cK \cap M$ is dense in $\cK$.
\item Let $P$ be a II$_1$ factor and $\bim{M}{\cH}{P}$ an $M$-$P$-bimodule. Suppose that $\cL \subset \cH$ is a closed $N$-$P$-subbimodule.
Denote by $\cK * \cL$ the closure of $(\cK \cap M) \cL$ inside $\cH$. Then, $\cK * \cL$ is an $N$-$P$-bimodule that is isomorphic to a subbimodule of $\cK \ot_N \cL$. Furthermore, whenever $\cK_0 \subset \cK \cap M$ is such that $\cK_0 \subset \cK$ is dense, also $\cK_0 \cL$ follows dense in $\cK * \cL$. If $\cK * \cL$ is non-zero and $\cK \ot_N \cL$ is irreducible, it follows that $\cK * \cL$ and $\cK \ot_N \cL$ are isomorphic $N$-$P$-bimodules.
\end{enumerate}
By symmetry, similar statements hold on the right. In particular, whenever $\bim{P}{\cH}{M}$ is a $P$-$M$-bimodule with closed $P$-$N$-subbimodule $\cL$, we define $\cL * \cK$ as the closure of $\cL (\cK \cap M)$ inside $\cH$ and find that $\cL * \cK$ is isomorphic with a $P$-$N$-subbimodule of $\cL \ot_N \cK$.
\end{lemm}
\begin{proof}
Choose a projection $p \in N^n$, a finite index inclusion $\vphi : N \recht pN^n p$ and an $N$-$N$-bimodular unitary
$$U : \bim{\vphi(N)}{p (\C^n \ot \rL^2(N))}{N} \recht \bim{N}{\cK}{N} \; .$$
Define $v_i \in \rL^2(M)$ by the formula $v_i := U(p(e_i \ot 1))$. Put $v := \sum_{i=1}^n e_i^* \ot v_i$, which belongs to $(\C^n)^* \ot \rL^2(M)$. By construction, $\lspan\{ v_i N \mid i=1,\ldots,n\}$ is dense in $\cK$ and $U(\xi) = v\xi$ for all $\xi \in p(\C^n \ot N)$. Since $U$ is $N$-$N$-bimodular, we have $a v = v \vphi(a)$ for all $a \in N$ and, in particular, $v = vp$.  It follows that $v v^*$ is an element of $\rL^1(M)$ commuting with $N$. By the irreducibility of $N \subset M$, we get $vv^* \in \C 1$. In particular, $v$ is bounded and $v_i \in M$ for every $i$.

Since $U(\xi) = v \xi$ for all $\xi \in p(\C^n \ot N)$, also $U^*(b) = (\id \ot E_N)(v^* b)$ for all $b \in M$. In particular, $U^*(M) \subset p(\C^n \ot N)$ and it follows that $\cK \cap M = U(p (\C^n \ot N)) = \lspan\{ v_i N \mid i=1,\ldots,n\}$. This proves the first part of the lemma.

To prove the second part, let $P$ be a II$_1$ factor and $\bim{M}{\cH}{P}$ an $M$-$P$-bimodule with closed $N$-$P$-subbimodule $\cL$.
Define $\cK * \cL$ as the closure of $(\cK \cap M) \cL$ inside $\cH$. Define the subspace $\cL_0 \subset \cL$ of vectors $\xi \in \cL$ such that the map $N \recht \cL : a \mapsto a \xi$ extends to a bounded operator from $\rL^2(N)$ to $\cL$. Then, $\cL_0$ is dense in $\cL$. Fix $\xi \in \cL_0$. We claim that the map $R : \cK \cap M \recht \cH : a \mapsto a \xi$ extends to a bounded operator from $\cK$ to $\cH$. Since $\xi \in \cL_0$, the map
$$S : \C^n \ot N \recht \cH : e_i \ot a \mapsto v_i a \xi$$
extends to a bounded operator from $\C^n \ot \rL^2(N)$ to $\cH$, that we still denote by $S$. By construction, $S(\vphi(a) \eta) = a S(\eta)$ for all $a \in N$ and $\eta \in \C^n \ot \rL^2(N)$. In particular, $S(\eta) = S(p\eta)$ for all $\eta \in \C^n \ot \rL^2(N)$. It follows that $R(U(\eta)) = S(\eta)$ for all $\eta \in p (\C^n \ot N)$. Since $U$ is unitary and $\cK \cap M$ equals $U(p(\C^n \ot N))$, the claim is proven.

Suppose now that $\cK_0 \subset \cK \cap M$ and that $\cK_0 \subset \cK$ is dense. From the claim in the previous paragraph, we get
$$\cK * \cL = \overline{(\cK \cap M) \cL} = \overline{(\cK \cap M) \cL_0} = \overline{\cK_0 \cL_0} = \overline{\cK_0 \cL} \; .$$
So, $\cK_0 \cL$ is dense in $\cK * \cL$.

The Connes tensor product $\cK \ot_N \cL$ can be realized as the $N$-$P$-bimodule $\bim{\vphi(N)}{p (\C^n \ot \cL)}{P}$. The linear operator $T : p (\C^n \ot \cL) \recht \cH : T(\xi) = v \xi$ is $N$-$P$-bimodular with range $\lspan\{v_i \cL \mid i=1,\ldots,n\}$. From the results above, it follows that the closure of the range of $T$ equals $\cK * \cL$. Taking the polar decomposition of $T$, we find an $N$-$P$-bimodular isometry of $\cK * \cL$ into $\cK \ot_N \cL$.
\end{proof}

\subsection{Quasi-normalizers} \label{sec.quasireg}

Let $(M,\tau)$ be a tracial von Neumann algebra and $N \subset M$ a von Neumann subalgebra.
\begin{itemize}
\item
The \emph{quasi-normalizer} of $N$ inside $M$ is defined as:
$$
\QN_M(N) = \Bigl\{a \in M  \; \Big | \;  \exists a_1,\dots, a_n,b_1,\dots,b_m \in M\ \text{s.t.}\ Na \subset
\sum_{i=1}^n a_i N\ \text{and}\ aN \subset \sum_{i=1}^m N b_i \Bigr \}\; .
$$
\item
The inclusion $N \subset M$ is called \emph{quasi-regular} if $\QN_M(N)'' = M$.
\end{itemize}
Remark that the quasi-normalizer of $N\subset M$ is a unital $*$-subalgebra of $M$ containing $N$.

Let $\Gamma$ be a group and $\Lambda \subset \Gamma$ a subgroup.
\begin{itemize}
\item
The \emph{commensurator} of $\Lambda \subset \Gamma$ is defined as
$$\Comm_\Gamma(\Lambda) :=\{ g \in \Gamma \mid g \Lambda g^{-1} \cap \Lambda\ \textrm{has finite index in}\ g\Lambda g^{-1} \textrm{and in}\ \Lambda \}\; .$$
\item
The inclusion $\Lambda \subset \Gamma$ is called \emph{almost normal} if $\Comm_\Gamma(\Lambda) = \Gamma$.
\end{itemize}

Remark that the inclusion $\rL(\Lambda) \subset \rL(\Gamma)$ is quasi-regular if and only if the inclusion $\Lambda \subset \Gamma$ is almost normal. A typical example of an almost normal subgroup is $\SL(n,\Z) \subset \SL(n,\Q)$.

Again we prove an elementary lemma for the convenience of the reader.

\begin{lemm} \label{lem.decompose-quasireg}
Let $N \subset M$ be an irreducible, quasi-regular inclusion of II$_1$ factors. Then, $\bim{N}{\rL^2(M)}{N}$ is the orthogonal direct sum of a family of irreducible, finite index $N$-$N$-subbimodules $K_i \subset \rL^2(M)$, $i \in I$. Writing $K_i^0 := K_i \cap M$, we have, for any choice of decomposition, $\lspan\{K_i^0 \mid i \in I\} = \QN_M(N)$.
\end{lemm}
\begin{proof}
Whenever $a \in \QN_M(N)$, the closure of $N a N$ inside $\rL^2(M)$ is an $N$-$N$-subbimodule of finite index. So, the linear span of all finite index $N$-$N$-subbimodules of $\rL^2(M)$, is dense in $\rL^2(M)$. Hence, $\rL^2(M)$ can be decomposed into an orthogonal direct sum of a family of irreducible, finite index $N$-$N$-subbimodules $K_i \subset \rL^2(M)$, $i \in I$. Write $K_i^0 := K_i \cap M$. By Lemma \ref{lem.product} and its right-handed analogue, we find $v^i_j$, $j = 1,\ldots,n_i$ and $w^i_j$, $j = 1,\ldots,m_i$ in $K_i^0$ such that
$$\lspan \{v^i_j N \mid j=1,\ldots,n_i\} = K_i^0 = \lspan \{N w^i_j \mid j = 1,\ldots,m_i\} \; .$$
Since $N K_i^0 N = K_i^0$, it follows that $K_i^0 \subset \QN_M(N)$.

The proof of Lemma \ref{lem.product} yields $v_i \in (\C^{n_i})^* \ot M$ such that the orthogonal projection $P_i$ of $\rL^2(M)$ onto $K_i$ satisfies $P_i(a) = v_i (\id \ot E_N)(v_i^* a)$ for all $a \in M$. In particular, $P_i(M) = K_i^0$.

Let now $a \in \QN_M(N)$ and decompose the closure of $N a N$ inside $\rL^2(M)$ as a direct sum of irreducible $N$-$N$-subbimodules $H_1,\ldots,H_n$. Lemma \ref{lem.finite-dim} implies that for every $i=1,\ldots,n$, there exists a finite subset $I_i \subset I$ such that $H_i \not\cong K_j$ whenever $j \in I \setminus I_i$. So, we find a finite subset $I_0 \subset I$ such that $N a N \subset \lspan \{K_i \mid i \in I_0\}$. But then,
$$a = \sum_{i \in I_0} P_i(a) \in \lspan \{K_i^0 \mid i \in I\} \; .$$
\end{proof}

\subsection{Amalgamated free products} \label{sec.amal}

We recall now some basic facts and notations about amalgamated free products, see \cite{popa-amal} and \cite{voi} for more details. Let $(M_0, \tau_0)$ and $(M_1,\tau_1)$ be tracial von Neumann algebras with a common von Neumann subalgebra $N$ such that ${\tau_0}_{|N}={\tau_1}_{|N}$. We denote by $E_i$ the unique trace preserving conditional expectation of $M_i$ onto $N$. The amalgamated free product $M_0 *_N M_1$ is, up to $E$-preserving isomorphism, the unique pair $(M,E)$ satisfying the following two conditions.
\begin{itemize}
\item The von Neumann algebra $M$ is generated by embeddings of $M_0$ and $M_1$ that are identical on $N$, and is equipped with a conditional expectation $E : M \droite N$.
\item The subalgebras $M_0$ and $M_1$ are free with amalgamation over $N$ with respect to $E$. This means that $E(x_1 \cdots x_n) = 0$ whenever $x_j \in M_{i_j}$ such that $E_{i_j}(x_j)=0$ and $i_1 \neq i_2$, \linebreak $i_2 \neq i_3, \ldots, i_{n-1} \neq i_n$.
\end{itemize}

The amalgamated free product $M_0 *_N M_1$ has a dense $*$-subalgebra given by
$$N \;\;\oplus\;\; \bigoplus_{n \geq 1} \;\; \left(
\bigoplus_{i_1 \neq i_2, \ldots, i_{n-1} \neq i_n} \overset{\circ}{M_{i_1}}
\cdots \overset{\circ}{M_{i_n}} \right)\; ,$$
where $\overset{\circ}{M_{i_k}} := M_{i_k}\ominus N$. The von Neumann algebra
$M_0 *_N M_1$ has a trace, defined by $\tau:={\tau_0}\circ E = {\tau_1}\circ E $.

\subsection{Minimal actions of compact groups and bimodule categories} \label{subsec.minimal}

We assume all compact groups to be second countable. A strongly continuous action $G \actson M$ of a compact group $G$ on a II$_1$ factor $M$ is said to be \emph{minimal} if the map $G \droite \Aut(M)$ is injective and if $M \cap (M^G)' = \C 1$. Here, $M^G$ is the von Neumann algebra of $G$-fixed points in $M$. We always denote by $H_\pi$ the Hilbert space of the representation $\pi$ and by $\eps$ the trivial representation.

Every minimal action $G \actson M$ gives rise to a tensor functor $\Rep(G) \recht \Bimod(M^G)$. This goes back to \cite{roberts} and can be stated explicitly as follows.

\begin{prop} \label{prop.tensorfunctor}
Let $G$ be a second countable compact group and $\si : G \actson M$ a minimal action. Set $P:= M^G$. Then,
\begin{align*}
& \Rep(G) \recht \Bimod(P) : \pi \mapsto \bim{P}{\Mor(H_\pi,\eL^2(M))}{P} \quad \text{where}\quad \langle S,T \rangle := \Tr(S^* T) \\ & \text{and} \quad (a \cdot S \cdot b)(\xi) = a S(\xi) b \quad\text{for all}\;\; S,T \in \Mor(H_\pi,\eL^2(M)) \; , \; a,b \in P \; , \xi \in H_\pi
\end{align*}
defines a fully faithful tensor functor from the category $\Rep(G)$ of finite dimensional unitary representation of $G$ to the category $\Bimod(P)$ of finite index $P$-$P$-bimodules.
\end{prop}

Let $\si: G \actson M$ be a minimal action and choose a complete set $\irr(G)$ of inequivalent, irreducible unitary representations of $G$. For every $\pi \in \irr(G)$, we choose and fix a unitary $V_\pi \in \B(H_\pi) \ot M$ satisfying $(\id \ot \si_g)(V_\pi) = V_\pi (\pi(g) \ot 1)$, see e.g.\ \cite[Theorem 12 and following comments]{Wass}.

Put $P:= M^G$. For every $\pi \in \irr(G)$, the map
\begin{equation} \label{eq. psi pi}
\psi_\pi : P \droite \B(H_\pi) \ot P : \psi_\pi(a) = V_\pi (1 \ot a)V_\pi^*
\end{equation}
defines an irreducible, finite index inclusion. Define the Hilbert space $H(\psi_\pi) = H_\pi^* \ot \eL^2(P)$, which is a $P$-$P$-bimodule as $\bim{P}{(H_\pi^* \ot \eL^2(P))}{\psi_\pi(P)}$. Of course, in the light of Proposition \ref{prop.tensorfunctor}, we have $H(\psi_\pi) \cong \Mor(H_\pi,\eL^2(M))$ as $P$-$P$-bimodules.

We introduce now some notations concerning \emph{spectral subspaces} of irreducible representations. Denote by $\Mor(\pi,M)$ the space of linear maps $S : H_\pi \recht M$ satisfying $\si_g \circ S = S \circ \pi(g)$. We denote the linear span of $\Mor(\pi,M) H_\pi$ as $L^0(\pi) \subset M$. The closure of $L^0(\pi)$ inside $\eL^2(M)$ is denoted by $L(\pi)$.
\begin{itemize}
\item As a $P$-$P$-bimodule, $L(\pi)$ is the direct sum of $\dim(\pi)$ copies of $H(\psi_\pi)$. More precisely, if you consider on $\B(H_\pi) \ot P$ the scalar product given by $\Tr \ot \tau$, the map
    $$\theta_\pi : \bim{{1 \ot P}}{\bigl( \B(H_\pi) \ot P \bigr)}{\psi_\pi(P)} \recht \bim{P}{L^0(\pi)}{P} : \theta_\pi(a) = \dim(\pi)^{1/2} (\Tr \ot \id)(a V_\pi)$$
    is $P$-$P$-bimodular, bijective and extends to an isometry $\B(H_\pi) \ot \eL^2(P) \hookrightarrow \eL^2(M)$.
\item The adjoint of $\theta_\pi$ is given by $E_\pi := \theta_\pi^*$ satisfying
\begin{equation}\label{eq.Epi}
E_\pi(b) = \dim(\pi)^{1/2} \int_G (\pi(g)^* \ot \si_g(b))V_\pi^* \; dg
\end{equation}
for all $b \in M$.
\item Since every unitary representation of $G$ splits as a direct sum of irreducibles, we have $$\sum_{\pi \in \irr(G)} E_\pi^* E_\pi = 1 \; .$$ Equivalently, $\eL^2(M)$ is the orthogonal direct sum of the subspaces $L(\pi)$, $\pi \in \irr(G)$.
\end{itemize}

\begin{rema} \label{rem. N in P quasi regular}
The coefficients of the unitaries $V_\pi$ quasi-normalize $M^G$ and so, the inclusion $M^G \subset M$ is quasi-regular. In fact, by Lemma \ref{lem.decompose-quasireg}, the quasi-normalizer of $M^G$ inside $M$ equals the linear span of all $L^0(\pi)$, $\pi \in \irr(G)$.
\end{rema}

For later use (in the proof of Lemma \ref{lem.step-4}), we record the following elementary property.

\begin{lemm} \label{lemm.pi ot eta}
Let $\si : G \actson M$ be a minimal action of a compact group $G$ on a II$_1$ factor $M$. Let $\pi$ and $\eta$ be irreducible representations of $G$. Take $\mu_1,\ldots,\mu_n \in \irr(G)$, with possible repetitions, and isometries $v_i \in \Mor(\mu_i,\pi \ot \eta)$ satisfying $\sum_{i=1}^n v_i v_i^* = 1$. There exist $X_i \in \B(H_{\mu_i}, H_\pi \ot H_\eta) \ot M^G$ with $X_i^* X_i = 1$ for all $i$ and $\sum_{i=1}^n X_i X_i^* = 1$ such that
$$(V_\pi)_{13} (V_\eta)_{23} = \sum_{i = 1}^n X_i V_{\mu_i} (v_i^* \ot 1) \; .$$
\end{lemm}

\subsection{Freeness and free products of fusion algebras} \label{subsec.freeness}

The notions of freeness and free product of fusion algebras were introduced in \cite[Section 1.2]{BJ}, in the study of free composition of subfactors. For convenience, we recall the definition.

\begin{defi}[{\cite[Section 1.2]{BJ}}] \label{def.freeness}
Let $\cA$ be a fusion algebra and $\cA_i \subset \cA$ fusion subalgebras for $i = 1,2$. We say that $\cA_1$ and $\cA_2$ are \emph{free inside $\cA$} if every alternating product of irreducibles in $\cA_i \setminus \{e\}$, remains irreducible and different from $\{e\}$.
\end{defi}

Given fusion algebras $\cA_1$ and $\cA_2$, there is up to isomorphism a unique fusion algebra $\cA$ generated by copies of $\cA_1$ and $\cA_2$ that are free. We call this unique $\cA$ the \emph{free product} of $\cA_1$ and $\cA_2$ and denote it by $\cA_1 * \cA_2$.

Denote by $R$ the hyperfinite II$_1$ factor. The fusion algebra $\FAlg(R)$ is huge, in the sense that $\FAlg(R)$ contains many free fusion subalgebras. More precisely, it was shown in Theorem 5.1 of \cite{V1} that countable fusion subalgebras of $\FAlg(R)$ can be made free by conjugating one of them with an automorphism of $R$ (see Theorem \ref{Gdelta} below). Note that the same result has first been proven for countable subgroups of $\Out(R)$ in \cite{IPP}. In both cases, the key ingredients come from \cite{P10}.

Let $M$ be a II$_1$ factor and $\bim{M}{\cK}{M}$ a finite index $M$-$M$-bimodule. Whenever $\al \in \Aut(M)$, we define the conjugation of $K$ by $\al$ as the bimodule $\cK^\al := H(\al^{-1}) \ot_M \cK \ot_M H(\al)$.

\begin{theo}[Theorem 5.1 in \cite{V1}] \label{Gdelta}
Let $R$ be the hyperfinite II$_1$ factor and $\cA_0,\cA_1$ two countable fusion
subalgebras of $\FAlg(R)$. Then,
$$\{\alpha \in \Aut(R) \mid  \cA_0^{\alpha}\ \text{and}\ \cA_1\ \text{are free}\}$$
is a $G_{\delta}$-dense subset of $\Aut(R)$.
\end{theo}

\subsection{Intertwining by bimodules}

In this section, we briefly recall Popa's \emph{intertwining-by-bimodules technique} introduced in \cite{P3}. This very powerful technique is used to deduce unitary conjugacy of two von Neumann subalgebras $A$ and $B$ of a tracial von Neumann algebra $(M, \tau)$ from their weak containment $A \prec_M B$ that we define now.

\begin{defi} \label{def.embed}
Let $(M,\tau)$ be a von Neumann algebra with normal faithful tracial state. Let $A,B \subset M^n$ be possibly non-unital embeddings. We write $A \underset{M}{\prec} B$ if $1_A (\M_n(\C) \ot \eL^2(M))1_B$ contains a non-zero $A$-$B$-subbimodule $K$ with $\dim(K_B) < \infty$.
\end{defi}

By \cite[Theorem 2.1]{P1} (see also Appendix~C in \cite{V-Bourbaki}), we have $A \prec_M B$ if and only if there exist $m \in \N$, a non-zero partial isometry $v \in 1_A \bigl( \C^n (\C^m \ot \C^n)^* \ot M \bigr)(1 \ot 1_B)$ and a possibly non-unital $*$-homomorphism $\psi : A \droite \M_m(\C) \ot B$  satisfying $a v = v \psi(a)$ for all $a \in A$. In particular, writing for $i=1,\ldots,m$, $v_i := v(e_i \ot 1 \ot 1)$, we have found $v_1,\ldots,v_m \in 1_A M^n 1_B$ such that $\lspan \{ v_i B \mid i=1,\ldots,m\}$ is a left $A$-module.

\subsection{Property (T) for II$_1$ factors} \label{sec.T}

Property (T) for finite von Neumann algebras was defined by Connes and Jones in \cite{CJ}: a II$_1$ factor $(M,\tau)$ has property (T) if and only if there exists $\epsilon >0$ and a finite subset $F\subset M$ such that every $M$-$M$-bimodule $H$ that has a unit vector $\xi$ satisfying $\|x \xi - \xi x\| \leq \epsilon$, for all $x \in F$, actually has a non-zero $M$-central vector $\xi_0$, meaning that $x \xi_0 = \xi_0 x$, for all $x \in M$. Note that an ICC group $\Gamma$ has property (T) if and only if the II$_1$ factor $\rL(\Gamma)$ has property (T) in the sense of Connes and Jones.

\subsection{Rigid subalgebras of amalgamated free products}

For the convenience of the reader we quote some of the results obtained by Ioana, Peterson and Popa in \cite[Theorems 1.1 and 5.1]{IPP}. We only state the version as needed in this paper and give a few comments.

\begin{theo}[Theorems 1.1 and 5.1 in \cite{IPP}] \label{thm.IPP}
Let $(M_0,\tau_0)$ and $(M_1,\tau_1)$ be von Neumann algebras with faithful normal tracial states. Let $N$ be a common von Neumann subalgebra of $M_0$ and $M_1$ with $\tau_0|_N = \tau_1|_N$. Define the amalgamated free product $M := M_0 *_N M_1$ with respect to the unique trace-preserving conditional expectations.
\begin{enumerate}
\item Let $p \in M_0^n$ be a projection and $Q \subset p M_0^n p$ a von Neumann subalgebra satisfying $Q \not\prec_{M_0} N$. Whenever $\cK \subset p(\C^n \ot \rL^2(M))$ is a $Q$-$M_0$-subbimodule with $\dim(\cK_{M_0}) < \infty$, we have $\cK \subset p(\C^n \ot \rL^2(M_0))$. In particular, the quasi-normalizer of $Q$ inside $p M^n p$ is contained in $p M_0^n p$.
\item Assume that $M_0$ is a factor. Let $p \in M^n$ be a projection and $Q \subset p M^n p$ a von Neumann subalgebra with the relative property (T) in the sense of Popa \cite[Definition 4.2]{P3} (which holds, in particular, if $Q$ is a II$_1$ factor having the property (T) of Connes and Jones explained in Section \ref{sec.T}). Assume that $Q \not\prec_{M} M_1$. Then, there exists $u \in M^n$ satisfying $uu^* = p$ and $u^* Q u \subset M_0^n$.
\end{enumerate}
\end{theo}

The first item is precisely \cite[Theorem 1.1]{IPP}. Note, in comparison with \cite{IPP}, that working with matrices over $M$ is not really more general, because we identify $M^n$ with $M_0^n *_{N^n} M_1^n$.

The second item is a special case of \cite[Theorem 5.1]{IPP}. Again, it is not a real generalization to work with matrices over $M$, so that we may assume $n = 1$. Next, since $M$ is a factor, we can choose a von Neumann subalgebra $\tilde{Q} \subset M$, containing $p$, such that $Q = p \tilde{Q} p$ and such that the inclusion $\tilde{Q} \subset M$ still has the relative property (T) and still satisfies $\tilde{Q} \not\prec_{M} M_1$. Then, \cite[Theorem 5.1]{IPP} provides a unitary $v \in M$ such that $v^* \tilde{Q} v \subset M_0$. We put $u = p v$. Finally, in \cite[Theorem 5.1]{IPP}, the extra assumption is made that the inclusions $N \subset M_i$ are homogeneous. But, as explained in \cite[Section 5.2]{Houd}, the homogeneity assumption is superfluous.

\section{Proof of Theorem \ref{thm.main}}

Fix a second countable compact group $G$ and an action $G \overset{\sigma}{\curvearrowright} M_1$ on the II$_1$ factor $M_1$. Denote $N = M_1^G$ and fix an inclusion $N \subset M_0$ into the II$_1$ factor $M_0$. We are interested in the II$_1$ factor

$$M: = M_0 *_N M_1$$

and extend the action $G \actson M_1$ to an action $G \actson M$ by acting trivially on $M_0$.

\subsubsection*{Assumptions}
\begin{enumerate}\renewcommand{\theenumii}{.\alph{enumii}}\renewcommand{\labelenumii}{\arabic{enumi}.\alph{enumii})}
\item\label{1} {\bf Assumption on the action $G \overset{\sigma}{\curvearrowright} M_1$~:} $\sigma$ is minimal and $M_1$ is hyperfinite.
\item\label{2} {\bf Assumptions on the inclusion $N \subset M_0$.}
\begin{enumerate}
\item\label{2A} The inclusion $N \subset M_0$ is irreducible, i.e.\ $N' \cap M_0 = \C 1$ and is quasi-regular (see Section \ref{sec.quasireg}).
\item\label{2B} A condition on absence of finite dimensional unitary representations (cf.\ Lemmas \ref{lem.absence1} and \ref{lem.absence2}). Denote by $\cF_0$ the fusion subalgebra of $\FAlg(N)$ generated by the finite index $N$-$N$-subbimodules of $\bim{N}{\rL^2(M_0)}{N}$. Whenever $\bim{M_0}{\cK}{M_0}$ is an irreducible finite index $M_0$-$M_0$-bimodule containing a non-zero element of $\cF_0$ as $N$-$N$-subbimodule, we have $\bim{M_0}{\cK}{M_0} \cong \bim{M_0}{\rL^2(M_0)}{M_0}$.
\end{enumerate}
\item\label{3} {\bf Rigidity assumption:} there exists $N \subset N_0 \subset M_0$ such that $N_0$ has property (T) in the sense of Connes and Jones \cite{CJ} and such that $N_0 \subset M_0$ is quasi-regular.
\item\label{4} {\bf Relation between $G \overset{\sigma}{\curvearrowright} M_1$ and $N \subset M_0$.} Denote by $\cF$ the fusion subalgebra of $\FAlg(N)$ generated by the finite index $N$-$N$-bimodules that arise as $N$-$N$-subbimodule of a finite index $M_0$-$M_0$-bimodule. Then, $\cF$ is free with respect to the canonical image of $\Rep(G)$ in $\FAlg(N)$, given by the minimal action $\sigma$ (see Proposition \ref{prop.tensorfunctor} and recall that $N = M_1^G$).
\end{enumerate}

\begin{rema}\label{rem.countable}
If $N \subset M_0$ is an irreducible inclusion of II$_1$ factors having the relative property (T) (so, in particular, if $N \subset M_0$ satisfies assumption \ref{3}), one can repeat the proof of \cite[Lemma 4.1]{V1} and it follows that the fusion algebra $\cF$ defined in assumption \ref{4} is countable.
\end{rema}

We now clarify the slightly mysterious assumption \ref{2B}. The first of the following two lemmas is not used in the paper, but makes the link with absence of finite dimensional unitary representations.

\begin{lemm} \label{lem.absence1}
Let $N$ be a II$_1$ factor and $\Gamma \actson N$ an action of the countable group $\Gamma$ by outer automorphisms of $N$. Put $M_0 := N \rtimes \Gamma$. Then, the inclusion $N \subset M_0$ satisfies assumption \ref{2B} if and only if $\Gamma$ has no non-trivial finite dimensional unitary representations.
\end{lemm}
\begin{proof}
Since $\Gamma \actson N$ is outer, the inclusion $N \subset M_0$ is irreducible. Denote by $(u_g)_{g \in \Gamma}$ the canonical unitaries in $M_0$. Since every $au_g$, $g \in \Gamma$, $a \in \cU(N)$, normalizes $N$, the inclusion $N \subset M_0$ is regular.

If $\pi : \Gamma \recht \cU(H)$ is a non-trivial irreducible finite dimensional unitary representation, define $\cK = H \ot \rL^2(M_0)$ with bimodule action given by
$$(a u_g) \cdot \xi = (\pi(g) \ot au_g) \xi \quad\text{and}\quad \xi \cdot b = \xi b \quad\text{for all}\;\; a \in N, g \in \Gamma, b \in M_0 \; .$$
Then, $\bim{M_0}{\cK}{M_0}$ is an irreducible finite index $M_0$-$M_0$-bimodule and $\bim{M_0}{\cK}{M_0} \not\cong \bim{M_0}{\rL^2(M_0)}{M_0}$. Nevertheless, we have the $N$-$N$-subbimodule $H \ot \rL^2(N) \subset \cK$, which is a sum of $\dim(H)$ copies of the trivial $N$-$N$-bimodule and hence, belongs to $\cF_0$. So, assumption \ref{2B} fails.

Conversely, suppose that $\Gamma$ has no non-trivial finite dimensional unitary representations and denote the action by $\Gamma \overset{\rho}{\actson} N$. Observe that the irreducible elements in $\cF_0$ are precisely the $N$-$N$-bimodules $H(\rho_s)$, $s \in \Gamma$, defined on the Hilbert space $\rL^2(N)$ with bimodule action
$$a \cdot \xi \cdot b := a \xi \rho_s(b) \quad\text{for all}\;\; \xi \in \rL^2(N), a,b \in N \; .$$
Let $\bim{M_0}{\cK}{M_0}$ be an irreducible finite index $M_0$-$M_0$-bimodule. Let $\cH \subset \cK$ be an $N$-$N$-subbimodule such that $\bim{N}{\cH}{N} \cong \bim{N}{H(\rho_s)}{N}$ for a certain $s \in \Gamma$. By irreducibility of $\cK$, it follows that the span of all $u_r \cdot \cH \cdot u_k$, $r,k \in \Gamma$, is dense in $\cK$. So, $\cK$ decomposes as a direct sum of $N$-$N$-subbimodules, each of them being isomorphic to one of the $H(\rho_r)$, $r \in \Gamma$. Multiplying on the right by $u_r$, it follows that every $N$-$M_0$-subbimodule of $\cK$ contains the trivial $N$-$N$-bimodule as an $N$-$N$-subbimodule.

Write $\bim{M_0}{\cK}{M_0} \cong \bim{\psi(M_0)}{p (\C^n \ot \rL^2(M_0))}{M_0}$. Since $\cK$ has finite index and $N \subset M_0$ is irreducible, the relative commutant $A:= \psi(N)' \cap p M_0^n p$ is finite dimensional. By the previous paragraph, we find for every minimal projection $q \in A$, a non-zero vector $\xi \in q (\C^n \ot \rL^2(M_0))$ satisfying $\psi(a) \xi = \xi a$ for all $a \in N$. As an element of $q(\M_n(\C) \ot \rL^1(M_0))q$, the operator $\xi \xi^*$ commutes with $\psi(N)$ and hence, is a multiple of $q$. So, we may assume that $\xi \in q(\C^n \ot M_0)$ and $\xi \xi^* = q$. Since $N \subset M_0$ is irreducible, $\xi^* \xi = 1$. Repeating this procedure for a family of minimal projections in $A$ summing to $1$, we find a unitary $X \in q(\C^n (\C^m)^* \ot M_0)$ such that $X^* \psi(a) X = 1 \ot a$ for all $a \in N$.

So, we may assume that $p = 1$ and $\psi(a) = 1 \ot a$ for all $a \in N$. But then, $\psi(u_g) = \pi(g) \ot u_g$ for all $g \in \Gamma$. Since $\Gamma$ has no finite dimensional unitary representations, it follows that $\pi(g) = 1$. By irreducibility of $\cK$, it follows that $n=1$ and $\bim{M_0}{\cK}{M_0} \cong \bim{M_0}{\rL^2(M_0)}{M_0}$.
\end{proof}

\begin{lemm} \label{lem.absence2}
Let $\Gamma$ be a countable group and $\Lambda < \Gamma$ an almost normal subgroup. Let $\Omega \in \rZ^2(\Gamma,S^1)$ be a scalar $2$-cocycle. Put $N := \rL_\Omega(\Lambda)$ and $M_0 := \rL_\Omega(\Gamma)$. If the following conditions hold, the inclusion $N \subset M_0$ satisfies assumption \ref{2B}.
\begin{itemize}
\item For all finite index subgroups $\Lambda_0 < \Lambda$, we have $\rL_\Omega(\Lambda_0)' \cap \rL_\Omega(\Gamma) = \C 1$.
\item The group $\Gamma$ has no non-trivial finite dimensional unitary representations.
\end{itemize}
\end{lemm}
\begin{proof}
Define the fusion subalgebra $\cF_0$ of $\FAlg(N)$ as in assumption \ref{2B}. We first claim that every $\bim{N}{\cH}{N}$ in $\cF_0$ \emph{is of good form,} by which we mean that there exists $n$, a finite index inclusion $\gamma : N \recht N^n$, elements $g_1,\ldots,g_n \in \Gamma$ and a finite index subgroup $\Lambda_0 < \Lambda \cap \bigcap_{i=1}^n g_i \Lambda g_i^{-1}$ such that $\bim{N}{\cH}{N} \cong \bim{\gamma(N)}{(\C^n \ot \rL^2(N))}{N}$ and
\begin{equation}\label{eq.ourclaim}
\gamma(u_h) = \sum_{i=1}^n e_{ii} \ot u_{g_i}^* u_h u_{g_i} \quad\text{for all}\;\; h \in \Lambda_0 \; .
\end{equation}
We now prove the following three statements. 1)~If $\bim{N}{\cH}{N}$ is of good form, then the same holds for all its $N$-$N$-subbimodules. 2)~If $\bim{N}{\cH}{N}$ and $\bim{N}{\cH'}{N}$ are both of good form, then $\cH' \ot_N \cH$ is again of good form. 3)~We can decompose $\rL^2(M_0)$ as a direct sum of $N$-$N$-subbimodules of good form.

It is obvious that being of good form is preserved by direct sums. Furthermore, because $\bim{N}{\rL^2(M_0)}{N}$ is isomorphic with its contragredient, $\cF_0$ is the smallest set of finite index $N$-$N$-bimodules (up to isomorphism), containing the finite index $N$-$N$-subbimodules of $\rL^2(M_0)$ and being stable under direct sums, tensor products and subbimodules. Hence, once statements 1, 2 and 3 are proven, our claim is proven as well.

Proof of 1). Let $\gamma$, $g_1,\ldots,g_n \in \Gamma$ and $\Lambda_0$ be given as above. If $h_1,\ldots,h_n \in \Lambda$ and if we replace $g_i$ by $g_i h_i$ and $\gamma$ by $a \mapsto W^* \gamma(a) W$, where $W \in N^n$ is given by $W := \sum_{i=1}^n e_{ii} \ot u_{h_i}$, formula \eqref{eq.ourclaim} still holds. So, we may assume that whenever $g_j \in g_i \Lambda$, then $g_j = g_i$. Let $s_1,\ldots,s_k$ be the distinct elements of $\Gamma$ such that $\{g_1,\ldots,g_n\} = \{s_1,\ldots,s_k\}$. By construction, $s_j \not\in s_i \Lambda$ whenever $i \neq j$. We have obvious diagonal projections $p_i \in \M_n(\C)$, summing to $1$ and satisfying
\begin{equation}\label{eq.pisi}
\gamma(u_h) = \sum_{i=1}^k p_i \ot u_{s_i}^* u_h u_{s_i} \; .
\end{equation}
Put $Z := \sum_{i=1}^k p_i \ot u_{s_i}$, which is a unitary in $M_0^n$. It follows that $\gamma(u_h) = Z^* (1 \ot u_h) Z$ for all $h \in \Lambda_0$ and hence,
\begin{equation}\label{eq.commutant}
\gamma(N)' \cap N^n \subset Z^* \bigl( (1 \ot \rL_\Omega(\Lambda_0))' \cap M_0^n \bigr) Z \subset Z^* (\M_n(\C) \ot 1) Z \; .
\end{equation}

Let now $q \in \gamma(N)' \cap N^n$ be a projection. We have to prove that the $N$-$N$-bimodule defined by $a \mapsto \gamma(a) q$ is again of good form. Because of \eqref{eq.pisi}, it suffices to prove that $q = r \ot 1$, where $r \in \M_n(\C)$ and $rp_i = p_i r$ for all $i =1,\ldots,k$. By \eqref{eq.commutant}, $q = Z^*(t \ot 1) Z$ for some $t \in \M_n(\C)$. Let $i \neq j$. Then, $(p_i \ot 1)q (p_j \ot 1)$ belongs to $\M_n(\C) \ot N$ and to $\M_n(\C) \ot \C u_{s_i^{-1} s_j}$. Hence, $(p_i \ot 1)q (p_j \ot 1) = 0$. But then,
$$q = \sum_{i=1}^k (p_i \ot 1)q (p_i \ot 1) = \sum_{i=1}^k p_i t p_i \ot 1 \; ,$$
concluding the proof of 1).

Proof of 2). Suppose that $\bim{N}{\cH}{N}$ and $\bim{N}{\cH'}{N}$ are of good form. Suppose that $\gamma : N \recht N^n$ defines $\bim{N}{\cH}{N}$ and satisfies \eqref{eq.ourclaim} with respect to $g_1,\ldots,g_n$, $\Lambda_0 < \Lambda \cap \bigcap_{i=1}^n g_i \Lambda g_i^{-1}$. Suppose that $\rho : N \recht N^m$ defines $\bim{N}{\cH'}{N}$ and satisfies \eqref{eq.ourclaim} with respect to $h_1,\ldots,h_m$ and $\Lambda_1 < \Lambda \cap \bigcap_{j=1}^m h_j \Lambda h_j^{-1}$. Then the composition $(\id \ot \gamma)\rho$ satisfies \eqref{eq.ourclaim} with respect to the group elements $h_j g_i$ and the subgroup $\Lambda_1 \cap \bigcap_{j=1}^m h_j \Lambda_0 h_j^{-1}$. It follows that $\cH' \ot_N \cH$ is of good form.

Proof of 3). For every $g \in \Gamma$, define $\cH(g)$ as the closure of $N u_g N$ in $\rL^2(M_0)$. Write
$$\Lambda g \Lambda = \bigsqcup_{i=1}^n g_i \Lambda \; .$$
The map $e_i \ot a \recht u_{g_i} a$ extends to a unitary $V : \C^n \ot \rL^2(N) \recht \cH(g)$ satisfying $V(\xi a) = V(\xi) a$ for all $\xi \in \C^n \ot \rL^2(N)$ and $a \in N$. Define $\gamma : N \recht N^n$ such that $V(\gamma(a) \xi) = a V(\xi)$ for all $\xi \in \C^n \ot \rL^2(N)$ and $a \in N$. One checks that $\gamma$ satisfies \eqref{eq.ourclaim} with respect to $g_1,\ldots,g_n$ and the subgroup $\Lambda_0 = \Lambda \cap \bigcap_{i=1}^n g_i \Lambda g_i^{-1}$. Since the linear span of all $\cH(g)$, $g \in \Gamma$, is dense in $\rL^2(M_0)$, the final statement 3) is proven and hence, also our claim.

Fix now an irreducible finite index $M_0$-$M_0$-bimodule $\bim{M_0}{\cK}{M_0}$ and let $\cH \subset \cK$ be an irreducible $N$-$N$-subbimodule such that $\bim{N}{\cH}{N}$ belongs to $\cF_0$. By Lemma \ref{lem.product}, for every $g,h \in \Gamma$, the closure of $N u_g \cdot \cH \cdot u_h N$ inside $\cK$ is isomorphic with an $N$-$N$-subbimodule of $\cH(g) \ot_N \cH \ot_N \cH(h)$. Hence, this closure belongs to $\cF_0$ and the irreducibility of $\cK$ implies that $\cK$ is a direct sum of $N$-$N$-subbimodules that all belong to $\cF_0$.

Write $\bim{M_0}{\cK}{M_0} \cong \bim{\psi(M_0)}{p(\C^m \ot \rL^2(M_0))}{M_0}$ for some finite index irreducible inclusion $\psi : M_0 \recht pM_0^m p$. Define $A := p M_0^m p \cap \psi(N)'$. Since $\cK$ is of finite index and $N \subset M_0$ is irreducible, it follows that $A$ is finite dimensional. Let $q \in A$ be a minimal projection. Because of the previous paragraph, we find a finite index inclusion $\gamma : N \recht N^n$ satisfying \eqref{eq.ourclaim} with respect to $g_1,\ldots,g_n$ and $\Lambda_0$ and a bimodular isometry
$$\theta : \bim{\gamma(N)}{(\C^n \ot \rL^2(N))}{N} \recht \bim{\psi(N)q}{q(\C^m \ot \rL^2(M_0))}{N} \; .$$
Define $\xi \in q(\C^m (\C^n)^* \ot \rL^2(M_0))$ by the formula
$$\xi := \sum_{i=1}^n \theta(e_i \ot 1) (e_i^* \ot 1) \; .$$
It follows that $\xi$ is non-zero and satisfies $\psi(a) \xi = \xi \gamma(a)$ for all $a \in N$. As an element of $q(\M_m(\C) \ot \rL^1(M_0))q$, the operator $\xi \xi^*$ commutes with $\psi(N)$. Hence, $\xi \xi^*$ is a multiple of $q$ and we may assume that $\xi \in q(\C^m(\C^n)^* \ot M_0)$ with $\xi \xi^* = q$. Define $V = \xi (\sum_i e_{ii} \ot u_{g_i}^*)$. It follows that $V \in q(\C^m (\C^n)^* \ot M_0)$, $VV^* = q$ and $\psi(u_h) V = V(1 \ot u_h)$ for all $h \in \Lambda_0$. Since $\rL_\Omega(\Lambda_0)$ has trivial relative commutant in $M_0$, it follows that $V^* V = p_1 \ot 1$ for some projection $p_1 \in \M_n(\C)$. Hence, we can find $W \in q(\C^m (\C^k)^* \ot M_0)$ satisfying $WW^* = q$, $W^* W = 1$ and $\psi(u_h) W = W (1 \ot u_h)$ for all $h$ in a finite index subgroup of $\Lambda$.

Repeating the same procedure for a set of minimal projections in $A$ summing to $1$ and taking the intersection of all finite index subgroups of $\Lambda$ involved, we find a unitary $X \in p(\C^m (\C^r)^* \ot M_0)$ and a finite index subgroup $\Lambda_0 < \Lambda$ satisfying $X^* \psi(u_h) X = 1 \ot u_h$ for all $h \in \Lambda_0$. So, we may actually assume that $\psi(u_h) = 1 \ot u_h$ for all $h \in \Lambda_0$.

If now $g \in \Gamma$, we get that $\psi(u_g) (1 \ot u_g^*)$ commutes with $1 \ot u_h$ for all $h \in \Lambda_0 \cap g \Lambda_0 g^{-1}$. Hence, $\psi(u_g) = \pi(g) \ot u_g$ for all $g \in \Gamma$, where $\pi : \Gamma \recht \cU(\C^r)$ is a finite dimensional unitary representation. By assumption, $\pi(g) = 1$ for all $g \in \Gamma$. So, $\psi(a) = 1 \ot a$ for all $a \in M_0$. By irreducibility of $\cK$, we get $r = 1$ and $\bim{M_0}{\cK}{M_0} \cong \bim{M_0}{\rL^2(M_0)}{M_0}$.
\end{proof}

In fact, we only use the following concrete example satisfying the conditions of Lemma \ref{lem.absence2} and hence providing an inclusion $N \subset M_0$ satisfying assumption \ref{2}, with $N$ being isomorphic with the hyperfinite II$_1$ factor. Moreover $N \subset M_0$ satisfies assumption \ref{3}.

\begin{exam} \label{exam.our-example}
Consider the group $\Gamma = (\Q^3 \oplus \Q^3) \rtimes \SL(3,\Q)$, defined by the action $A \cdot (x,y) = (Ax , (A^t)^{-1} y)$ of $\SL(3,\Q)$ on $\Q^3 \oplus \Q^3$. Choose an irrational number $\al \in \R$ and define $\Omega \in \rZ^2(\Gamma,S^1)$ such that
\begin{align*}
    \Omega\bigl( (x,y),(x',y') \bigr) &= \exp\bigl(i\al (\langle x,y'\rangle - \langle y, x'\rangle)\bigr) \qquad\text{for all}\;\; (x,y), (x',y') \in \Q^3 \oplus \Q^3 \; , \\
    \Omega(g,A) &= \Omega(A,g) = \Omega(A,B) = 1 \qquad\text{for all}\;\; g \in \Gamma \; , \; A,B \in \SL(3,\Q) \; .
\end{align*}
Set $\Lambda = \Z^3 \oplus \Z^3$. We define $N := \rL_\Omega(\Lambda)$ and $M_0 := \rL_\Omega(\Gamma)$. We prove that $N \subset M_0$ satisfies assumptions \ref{2} and \ref{3}.

Assumption \ref{3} follows by taking $N_0 := \rL_\Omega\bigl( (\Z^3 \oplus \Z^3) \rtimes \SL(3,\Z) \bigr)$, which has property (T) because $(\Z^3 \oplus \Z^3) \rtimes \SL(3,\Z)$ is a property (T) group.

Since $\SL(3,\Q)$ has no non-trivial finite-dimensional unitary representations and since the smallest normal subgroup of $\Gamma$ containing $\SL(3,\Q)$ is the whole of $\Gamma$, it follows that $\Gamma$ has no non-trivial finite-dimensional unitary representations. Because of Lemma \ref{lem.absence2}, it remains to prove that for every finite index subgroup $\Lambda_0 < \Lambda$, we have $\rL_\Omega(\Gamma) \cap \rL_\Omega(\Lambda_0)' = \C 1$. Write $\Lambda_1 = \Q^3 \oplus \Q^3$. Take $a \in \rL_\Omega(\Gamma) \cap \rL_\Omega(\Lambda_0)'$ and write, with $\eL^2$-convergence, $a = \sum_{g \in \Gamma} a_g u_g$. Since
$$\{h g h^{-1} \mid h \in \Lambda_0 \}$$
is infinite for all $g \in \Gamma - \Lambda_1$, we immediately get $a_g = 0$ for all $g \in \Gamma - \Lambda_1$. On the other hand, define
$$\pi : \Lambda_1 \recht \widehat{\Lambda_1} : h \mapsto \pi_h \quad\text{where}\;\; \pi_h(g) = \Omega(h,g)^{-2} \; .$$
It follows that $a_g \pi_g(h) = a_g$ for all $g \in \Lambda_1$ and all $h \in \Lambda_0$. If $g \in \Lambda_1 - \{0\}$, the character $\pi_g$ is not identically $1$ on $\Lambda_0$ and hence, $a_g = 0$. It follows that $a \in \C 1$.
\end{exam}

We deduce Theorem \ref{thm.main} from the following general statement.

\begin{theo}\label{thm.thm}
Under the assumptions at the beginning of the section and writing $M = M_0 *_N M_1$, the action $G \actson M$ is minimal and the natural tensor functor defined in \ref{prop.tensorfunctor}
$$\Rep(G) \droite \Bimod(M^G) : \pi \mapsto \bim{{M^G}}{\Mor(H_\pi,\eL^2(M))}{M^G}$$
is an equivalence of categories.
\end{theo}

The rest of this section is devoted to a proof of Theorem \ref{thm.thm} and deducing \ref{thm.main} as a corollary. Denote $P = M^G$ and make throughout the assumptions made at the beginning of the section.

Choose a complete set $\irr(G)$ of representatives for the set of irreducible unitary representations of $G$ modulo unitary conjugacy.
Since $G \overset{\si}{\actson} M_1$ is minimal, choose, for every $\pi \in \irr(G)$, a unitary $V_\pi \in \B(H_\pi) \ot M_1$ satisfying $(\id \ot \si_g)(V_\pi) = V_\pi (\pi(g) \ot 1)$ for all $g \in G$. Define the finite index inclusions
$$\psi_\pi : P \recht \B(H_\pi) \ot P : \psi_\pi(a) = V_\pi(1 \ot a)V_\pi^* \quad\text{for all}\;\; a \in P$$
and note that $\psi_\pi(N) \subset \B(H_\pi) \ot N$. As before, we denote by $H(\psi_\pi)$ the $N$-$N$-bimodule given by $\bim{N}{(H_\pi^* \ot \eL^2(N))}{{\psi_\pi(N)}}$.

\begin{rema} \label{rem.decomposition}
In Section \ref{subsec.minimal}, we explained that the $N$-$N$-bimodule $\bim{N}{\eL^2(M_1)}{N}$ can be decomposed into a direct sum of
$N$-$N$-bimodules $L(\pi)$, $\pi \in \irr(G)$, where $L(\pi)$ is isomorphic with $\dim(\pi)$ copies of $H(\psi_\pi)$. Furthermore $L(\pi)$ is the closure of $L^0(\pi) \subset M_1$, in such a way that the linear span of all $L^0(\pi)$ is an ultraweakly dense $*$-subalgebra of $M_1$.

Since $N \subset M_0$ is irreducible and quasi-regular, we can decompose $\rL^2(M_0) \ominus \rL^2(N)$ as the orthogonal direct sum of irreducible, non-trivial, finite index $N$-$N$-subbimodules $R_i, i \in I$. Put $R^0_i := R_i \cap M_0$. By Lemma \ref{lem.product}, $R^0_i$ is dense in $R_i$. By construction, $R^0_i \subset M_0 \ominus N$. Assume $0 \not\in I$ and put $R_0^0 := N$. By Lemma \ref{lem.decompose-quasireg}, $\lspan\{R_i^0 \mid i \in I \cup \{0\}\}$ equals $\QN_{M_0}(N)$ and is, in particular, an ultraweakly dense $*$-subalgebra of $M_0$.

Whenever $n \in \N \cup \{0\}$, $i_0,i_n \in I \cup \{0\}$, $i_1,\ldots,i_{n-1} \in I$ and $\pi_1,\ldots,\pi_n \in \irr(G) \setminus \{\eps\}$, denote by $R(i_0,\pi_1,\ldots,\pi_n,i_n)$ the closure of
$$R(i_0,\pi_1,\ldots,\pi_n,i_n)^0 := R^0_{i_0} L^0(\pi_1) R^0_{i_1} \cdots R^0_{i_{n-1}} L^0(\pi_n) R^0_{i_n}$$
inside $\rL^2(M)$. The definition of the amalgamated free product implies that $\rL^2(M)$ is the orthogonal direct sum of the subspaces $R(i_0,\pi_1,\ldots,\pi_n,i_n)$. Furthermore, the freeness with amalgamation implies that the map
$$a_0 \ot b_1 \ot a_1 \ot \cdots \ot b_n \ot a_n \mapsto a_0 b_1 a_1 \cdots b_n a_n$$
extends to a unitary from $R_{i_0} \ot_N L(\pi_1) \ot_N \cdots \ot_N L(\pi_n) \ot_N R_{i_n}$ onto $R(i_0,\pi_1,\ldots,\pi_n,i_n)$. In particular, $R(i_0,\pi_1,\ldots,\pi_n,i_n)$ is isomorphic with $\dim(\pi_1) \cdots \dim(\pi_n)$ copies of
\begin{equation}\label{eq.tussenmodule}
R_{i_0} \ot_N H(\psi_{\pi_1}) \ot_N \cdots \ot_N H(\psi_{\pi_n}) \ot_N R_{i_n} \; .
\end{equation}
Finally, the linear span of all $R(i_0,\pi_1,\ldots,\pi_n,i_n)^0$ is an ultraweakly dense $*$-subalgebra of $M$.

Denote by $\cF_0$ the fusion subalgebra of $\FAlg(N)$ generated by the finite index $N$-$N$-subbimodules $R_i, i \in I$.
Assumption \ref{4} implies in particular that the fusion subalgebra of $\FAlg(N)$ generated by $H(\psi_\pi)$, $\pi \in \irr(G)$, is free w.r.t.\ $\cF_0$. Therefore the $N$-$N$-bimodules appearing in \eqref{eq.tussenmodule} are irreducible and we have found a decomposition of $\bim{N}{\eL^2(M)}{N}$ as a direct sum of irreducible finite index $N$-$N$-bimodules. The trivial $N$-$N$-bimodule appears with multiplicity one in $\eL^2(M)$. This means that $N' \cap M = \C1$. In particular, the action $G \actson M$ is minimal.

The action of $G$ on $\rL^2(M)$ leaves every $R(i_0,\pi_1,\ldots,\pi_n,i_n)$ globally invariant and we denote by $R(i_0,\pi_1,\ldots,\pi_n,i_n)^G$ the subspace of $G$-fixed vectors. It follows that $\rL^2(P)$ is the orthogonal direct sum of the $N$-$N$-subbimodules $R(i_0,\pi_1,\ldots,\pi_n,i_n)^G$, which are, as $N$-$N$-bimodules isomorphic with a multiple of the $N$-$N$-bimodule given by \eqref{eq.tussenmodule}. This multiple is in its turn given by the dimension of the space of $(\pi_1 \ot \cdots \ot \pi_n)(G)$-invariant vectors in $H_{\pi_1} \ot \cdots \ot H_{\pi_n}$.
\end{rema}

\begin{lemm} \label{lem.step-1}
If $\bim{{M_0}}{\cH}{P}$ is an irreducible non-zero $M_0$-$P$-bimodule with $\dim(\cH_P) < \infty$, there exists $\eta \in \irr(G)$ and a non-zero $M_0$-$\psi_\eta(M_0)$-subbimodule $\cK \subset H_\eta^* \ot \cH$ with the following properties.
\begin{itemize}
\item $\dim(\cK_{\psi_\eta(M_0)}) < \infty$.
\item If $\rho \in \irr(G)$ and $\cL \subset H_\rho^* \ot \cH$ is a non-zero $M_0$-$\psi_{\rho}(M_0)$-subbimodule with $\dim(\lmo{{M_0}}{\cL}) < \infty$, then $\rho = \eta$ and $\cL \subset \cK$.
\end{itemize}
\end{lemm}
\begin{proof}
Take $\psi : M_0 \recht p P^n p$ such that $\bim{{M_0}}{\cH}{P} \cong \bim{{\psi(M_0)}}{p(\C^n \ot \eL^2(P))}{P}$. By assumption, the inclusion $\psi(M_0) \subset p P^n p$ is irreducible. From assumption \ref{3}, we get the property (T) II$_1$ factor $N_0$ and hence, $\psi(N_0)$ has property (T) and is a subalgebra of $p P^n p \subset p M^n p$. Recall that $M = M_0 *_N M_1$ and that $M_1$ is hyperfinite. Since there is no non-zero homomorphism from a property (T) II$_1$ factor to the hyperfinite II$_1$ factor, \ref{thm.IPP}.2 provides $u \in p(\C^n (\C^k)^* \ot M)$ with $uu^* = p$, $q:= u^*u \in M_0^k$ and $u^* \psi(N_0) u \subset q M_0^k q$. Since $N_0 \subset M_0$ is quasi-regular, \ref{thm.IPP}.1 implies that $u^* \psi(M_0) u \subset q M_0^k q$.

Define $\gamma : M_0 \recht q M_0^k q : \gamma(a) = u^* \psi(a) u$. We now use the bimodule maps $E_\pi$ given by \eqref{eq.Epi}. Take $\eta \in \irr(G)$ such that $(\id \ot E_\eta)(u) \neq 0$. So, we get a non-zero $v \in p(\C^n (\C^k \ot H_\eta)^* \ot P)$ satisfying $$\psi(a) v = v (\id \ot \psi_\eta)\gamma(a)$$ for all $a \in M_0$. Replacing $v$ by its polar part, we may assume that $v$ is a partial isometry. The irreducibility of $\psi(M_0) \subset p P^n p$ ensures that $vv^* = p$.

Define the $\psi(M_0)$-$\psi_\eta(M_0)$-subbimodule $\cK$ of $p(\C^n H_\eta^* \ot \eL^2(P))$ as the closure of $v(\C^k \ot \psi_\eta(M_0))$. By construction, $\dim(\cK_{\psi_\eta(M_0)}) < \infty$.

Let $\rho \in \irr(G)$ and let $\cL \subset p(\C^n H_\rho^* \ot \eL^2(P))$ be a non-zero $\psi(M_0)$-$\psi_\rho(M_0)$-subbimodule with $\dim(\lmo{{\psi(M_0)}}{\cL}) < \infty$. We have to prove that $\rho = \eta$ and $\cL \subset \cK$. Since $\dim(\lmo{{\psi(M_0)}}{\cL}) < \infty$, we can take a non-zero vector
$$\xi \in (1 \ot p)((\C^l \ot \C^n)H_\rho^* \ot \eL^2(P))$$
and a, possibly non-unital, $*$-homomorphism $\theta : M_0 \recht M_0^l$ satisfying $\xi \psi_{\rho}(a) = (\id \ot \psi)\theta(a) \xi$ for all $a \in M_0$ and such that $\cL$ is the closed linear span of $((\C^l)^* \ot \psi(M_0))\xi$.

Put $\zeta = (1 \ot 1 \ot V_\eta^*)(1 \ot v^*) \xi V_\rho$. Since $vv^* = p$, we know that $\zeta$ is non-zero. Then, $$\zeta \in (\C^l \ot \C^k \ot H_\eta)H_\rho^* \ot \eL^2(M)$$ and $\zeta$ satisfies $\zeta (1 \ot a) = ((\id \ot \gamma)\theta(a))_{124} \zeta$ for all $a \in M_0$. By Theorem \ref{thm.IPP}.1, it follows that $\zeta \in (\C^l \ot \C^k \ot H_\eta)H_\rho^* \ot \eL^2(M_0)$. In particular, $\zeta$ is $G$-invariant. Since $\zeta = (1 \ot 1 \ot V_\eta^*)(1 \ot v^*) \xi V_\rho$ and $\xi$ is a non-zero $G$-invariant vector, it follows that $\eta = \rho$ and $\zeta = (\zeta_0)_{124}$ for some $\zeta_0 \in \C^l \ot \C^k \ot \eL^2(M_0)$. It finally follows that
$$\xi = (1 \ot v)(\id \ot \id \ot \psi_\eta)(\zeta_0)$$
which belongs to $\C^l \ot \cK$, ending the proof of the lemma.
\end{proof}

\begin{lemm} \label{lem.step-2}
Let $\bim{P}{\cH}{P}$ be a finite index $P$-$P$-bimodule. For every non-zero irreducible $M_0$-$P$-subbimodule $\cH_0 \subset \cH$,
there exists $\eta \in \irr(G)$ and a non-zero $M_0$-$\psi_\eta(M_0)$-subbimodule $\cK \subset H_\eta^* \ot \cH_0$ such that $\bim{{M_0}}{\cK}{\psi_\eta(M_0)}$ has finite index.
\end{lemm}
\begin{proof}
For every $\pi \in \irr(G)$, define the finite index $P$-$P$-bimodule $\cH^\pi$ given by $\bim{{\psi_\pi(P)}}{(H_\pi \ot \cH)}{P}$. Since $M_0 \subset P$ is irreducible, we find for every $\pi \in \irr(G)$, a finite number $n_\pi$ and an orthogonal decomposition $\cH^\pi = \bigoplus_{i=1}^{n_\pi} \cH^{\pi,i}$ of $\cH^\pi$ into irreducible $M_0$-$P$-bimodules. For every $\pi,i$, take $\eta_{\pi,i} \in \irr(G)$ and a $\psi_\pi(M_0)$-$\psi_{\eta_{\pi,i}}(M_0)$-subbimodule $\cK^{\pi,i} \subset H_{\eta_{\pi,i}}^* \ot \cH^{\pi,i}$ satisfying the conclusions of Lemma \ref{lem.step-1}. Note that $\cK^{\pi,i} \subset H_\pi H_{\eta_{\pi,i}}^* \ot \cH$.

Define the subset $J \subset \irr(G) \times \irr(G)$ consisting of $(\pi,\eta)$ for which there exists $1 \leq i \leq n_\pi$ with $\eta_{\pi,i} = \eta$. Moreover, define $\cK^{\pi,\eta} = \lspan \{\cK^{\pi,i} \mid \eta_{\pi,i} = \eta\}$. By construction, $\cK^{\pi,\eta}$ is a non-zero $\psi_\pi(M_0)$-$\psi_\eta(M_0)$-subbimodule of $H_\pi H_\eta^* \ot \cH$, of finite right $\psi_\eta(M_0)$-dimension. Moreover, whenever $\pi,\eta \in \irr(G)$ and $\cK \subset H_\pi H_\eta^* \ot \cH$ is a $\psi_\pi(M_0)$-$\psi_\eta(M_0)$-subbimodule of finite left $\psi_\pi(M_0)$-dimension, it follows that $(\pi,\eta) \in J$ and $\cK \subset \cK^{\pi,\eta}$.

By symmetry, we also find a subset $J' \subset \irr(G) \times \irr(G)$ and for all $(\pi,\eta) \in J'$ a $\psi_\pi(M_0)$-$\psi_\eta(M_0)$-subbimodule $\cL^{\pi,\eta}$ of $H_\pi H_\eta^* \ot \cH$ which is of finite left $\psi_\pi(M_0)$-dimension and which has the following property: if $\pi,\eta \in \irr(G)$ and $\cL \subset H_\pi H_\eta^* \ot \cH$ is a non-zero $\psi_\pi(M_0)$-$\psi_\eta(M_0)$-subbimodule of finite right $\psi_\eta(M_0)$-dimension, we have $(\pi,\eta) \in J'$ and $\cL \subset \cL^{\pi,\eta}$.

But then, $J = J'$ and $\cK^{\pi,\eta} = \cL^{\pi,\eta}$ for all $(\pi,\eta) \in J = J'$. Hence, all $\cK^{\pi,\eta}$ are finite index $\psi_\pi(M_0)$-$\psi_\eta(M_0)$-bimodules. To conclude the proof of the lemma, it suffices to observe that $\bigoplus_{i=1}^{n_\epsilon} \cH^{\epsilon,i}$ is a decomposition of $\cH$ into irreducible $M_0$-$P$-subbimodules and that $\cK^{\epsilon,i} \subset H_{\eta_{\epsilon,i}}^* \ot \cH^{\epsilon,i}$ is the required finite index $M_0$-$\psi_{\eta_{\epsilon,i}}(M_0)$-subbimodule.
\end{proof}

\begin{lemm} \label{lem.step-3}
Let $\bim{P}{\cH}{P}$ be a finite index $P$-$P$-bimodule and $\cK \subset \cH$ a non-zero irreducible $M_0$-$M_0$-subbimodule such that $\bim{{M_0}}{\cK}{M_0}$ has finite index. Then, $\cK$ is the trivial $M_0$-$M_0$-bimodule: $\bim{M_0}{\cK}{M_0} \cong \bim{M_0}{\eL^2(M_0)}{M_0}$.
\end{lemm}
\begin{proof}
We may assume that $\bim{P}{\cH}{P}$ is irreducible.

{\bf Step 1.} $\cK$ contains a non-zero $N$-$N$-subbimodule $\cL$ with $\dim(\cL_N) < \infty$.

To prove step 1, take a finite index inclusion $\psi : P \recht pP^n p$ such that $\bim{P}{\cH}{P} \cong \bim{{\psi(P)}}{p (\C^n \ot \eL^2(P))}{P}$. Let $\cK \subset p(\C^n \ot \eL^2(P))$ be an irreducible non-zero $\psi(M_0)$-$M_0$-subbimodule such that $\bim{\psi(M_0)}{\cK}{M_0}$ has finite index. Take a finite index, irreducible inclusion $\rho : M_0 \recht q M_0^k q$ and a unitary
$$\theta : q(\C^k \ot \eL^2(M_0)) \recht \cK \quad\text{s.t.}\quad \theta(\rho(a) \mu b) = \psi(a) \theta(\mu) b \quad\text{for all}\;\; \mu \in q(\C^k \ot \eL^2(M_0)) \; , \; a,b \in M_0 \; .$$
Define the non-zero vector $\xi \in p(\C^n(\C^k)^* \ot \eL^2(P))q$ by
$$\xi = \sum_{i=1}^k \theta\bigl(q (e_i \ot 1)\bigr)\, (e_i^* \ot 1) \; .$$
It follows that $\psi(a) \xi = \xi \rho(a)$ for all $a \in M_0$. As an element of $p(\M_n(\C) \ot \rL^1(P))p$, the operator $\xi \xi^*$ commutes with $\psi(M_0)$. Since $M_0 \subset P$ is irreducible and $\psi(P) \subset p P^n p$ has finite index, the relative commutant $\psi(M_0)' \cap p P^n p$ is finite dimensional. Hence, $\xi \xi^*$ is bounded and it follows that $\xi \in p(\C^n (\C^k)^* \ot P)q$. The irreducibility of $\rho(M_0) \subset qM_0^k q$ implies that $E_{M_0}(\xi^* \xi)$ is a non-zero multiple of $q$. Denote by $v \in p(\C^n (\C^k)^* \ot P)q$ the polar part of $\xi$. Note that $\psi(a) v = v \rho(a)$ for all $a \in M_0$.

We claim that $\rho(N) \prec_{M_0} N$. Suppose not. Then, Theorem \ref{thm.IPP}.1 implies that the quasi-normalizer of $\rho(N)$ inside $qM^k q$ is contained in $q M_0^k q$. Since $N \subset P$ is quasi-regular, it follows that $v^* \psi(P) v \subset q M_0^k q$. Denote by $A$ the von Neumann subalgebra of $p P^n p$ generated by $\psi(P)$ and $vv^*$. Write $q_1 = v^*v$ and note that $q_1 \in qM_0^k q$. Since $\psi(P) \subset A \subset p P^n p$, it follows that $A \subset p P^n p$ has finite index. But then, $v^* A v$ is a finite index subalgebra of $q_1 P^k q_1$. Since $v^* A v \subset q_1 M_0^k q_1$ and $M_0 \subset P$ has infinite index, we have reached a contradiction. This proves the claim.

The claim and the remark following Definition \ref{def.embed} yield $b_1,\ldots,b_m \in q (\C^k \ot M_0)$ such that writing $V = \lspan\{b_i N \mid i = 1,\ldots,m\}$, we have $V \neq \{0\}$ and $\rho(N) V = V$. Define $\cL \subset \cK$ as the closure of $\lspan\{\xi b_i N \mid i=1,\ldots,m\}$. By construction $\cL$ is a $\psi(N)$-$N$-subbimodule of $\cK$ with $\dim(\cL_N) < \infty$. Since $E_{M_0}(\xi^* \xi)$ is a multiple of $q$, it also follows that $\cL$ is non-zero. So, Step~1 is proven.

{\bf Step 2.} $\cK$ is a direct sum of non-zero $N$-$N$-subbimodules $\cL$ such that $\bim{N}{\cL}{N}$ has finite index.

By Step 1, take a non-zero $N$-$N$-subbimodule $\cL^0$ of $\cK$ with $\dim(\cL^0_N) < \infty$. For all $a,b \in \QN_{M_0}(N)$, the closure of $Na \cdot \cL^0 \cdot bN$ is still an $N$-$N$-subbimodule of $\cK$ of finite right $N$-dimension. By irreducibility of $\bim{{M_0}}{\cK}{M_0}$ and quasi-regularity of $N \subset M_0$, the linear span of all $Na \cdot \cL^0 \cdot bN$ is dense in $\cK$. So, we have written $\cK$ as a direct sum of $N$-$N$-subbimodules of finite right $N$-dimension. By symmetry, we can also write $\cK$ as a direct sum of $N$-$N$-subbimodules of finite left $N$-dimension. Taking all non-zero intersections of $N$-$N$-subbimodules of both kinds, we end the proof of Step~2.

{\bf Step 3.} Define as above the fusion subalgebra $\cF_0$ of $\FAlg(N)$ generated by the finite index $N$-$N$-subbimodules of $\eL^2(M_0)$. We now prove that every irreducible $N$-$N$-subbimodule of $\cK$ belongs to $\cF_0$.

Choose an infinite set $K_i$, $i \in I$, of irreducible, non-trivial, inequivalent $N$-$N$-bimodules that appear as $N$-$N$-subbimodules of $\rL^2(M_0)$. Since $M_0$ is not hyperfinite, the inclusion $N \subset M_0$ has infinite index and hence, Lemma \ref{lem.finite-dim} implies that such an infinite set $I$ can be chosen. Denote $K_i^0 := K_i \cap M_0$. By Lemma \ref{lem.product}, $K_i^0$ is dense in $K_i$ and $K_i^0$ is finitely generated, both as a left and as a right $N$-module.

Assume by contradiction that $\cL \subset \cK$ is an irreducible $N$-$N$-subbimodule such that $\bim{N}{\cL}{N}$ has finite index and $\cL \not\in \cF_0$. Define $\cF$ as in assumption \ref{4} and note that by construction $\cL \in \cF$.
Take some $\pi \in \irr(G)$, $\pi \neq \epsilon$. Take $\eta \in \irr(G)$ unitarily equivalent with the contragredient representation of $\pi$. Let $\xi_0 \in H_\eta \ot H_\pi$ be a non-zero $(\eta \ot \pi)$-invariant vector. Define for every $i \in I$, the subspace $\cT_i \subset P$ given by
$$\cT_i := \lspan \bigl((H_\eta \ot H_\pi)^* \ot N\bigr) (V_\eta)_{13} (1 \ot 1 \ot K_i^0) (V_\pi)_{23}(\xi_0 \ot 1) \; .$$
It follows from Remark \ref{rem.decomposition} that the closure of $\cT_i$ in $\eL^2(P)$ is an $N$-$N$-bimodule isomorphic with $H(\psi_\eta) \ot_N K_i \ot_N H(\psi_\pi)$. Note for later use that the ultraweak closure of $P \cT_i$ is of the form $P q$ for some non-zero projection $q \in P$ that commutes with $N$ and hence equals $1$. So, $P \cT_i$ is ultraweakly dense in $P$.

We claim that the subspaces of $\cH$ defined by $\cT_i \cdot \cL \cdot P$, $i \in I$, are non-zero and mutually orthogonal in $\cH$. Once this claim is proven, we have found inside $\cH$ infinitely many orthogonal, non-zero $N$-$P$-subbimodules. This is a contradiction with $\bim{P}{\cH}{P}$ being of finite index and $N \subset P$ being irreducible. So, to conclude step 3, it remains to prove the claim.

Fix $i \in I$. Since $P \cT_i$ is ultraweakly dense in $P$, it follows that $\cT_i \cdot \cL \cdot P$ is non-zero. As at the end of Remark \ref{rem.decomposition}, decompose $\rL^2(P)$ as a direct sum of irreducible, finite index $N$-$N$-subbimodules $H_k$ such that, writing $H^0_k := H_k \cap P$, the linear span of all $H^0_k$ is an ultraweakly dense $*$-subalgebra of $P$. Take $i \neq j$ and take $k,l$. We have to prove that $\cT_i \cdot \cL \cdot H^0_k$ is orthogonal to $\cT_j \cdot \cL \cdot H^0_l$. It suffices to prove that their closures are disjoint $N$-$N$-bimodules.

By Lemma \ref{lem.product} and our description above of the closure of $\cT_i$, the closure of $\cT_i \cdot \cL \cdot H^0_k$ is isomorphic with an $N$-$N$-subbimodule of
$$H(\psi_\eta) \ot_N K_i \ot_N H(\psi_\pi) \ot_N \cL \ot_N H_k \; .$$
Since $\cL \in \cF \setminus \cF_0$, Frobenius reciprocity combined with Remark \ref{rem.decomposition}, implies that $\cL \ot_N H_k$ is isomorphic with a direct sum of $N$-$N$-bimodules of the form
$$\cL' \ot_N H(\psi_{\pi_1}) \ot_N R_{1} \ot_N \cdots \ot_N R_{{n-1}} \ot_N H(\psi_{\pi_n}) \ot_N R_{n} \; ,$$
where $\cL' \in \cF \setminus \cF_0$, $n \in \N \cup \{0\}$, $\pi_1,\ldots,\pi_n \in \irr(G) \setminus \{\eps\}$, $R_1,\ldots,R_n$ are irreducible $N$-$N$-subbimodules of $\rL^2(M_0)$ and $R_1,\ldots,R_{n-1}$ are non-trivial. Hence, the closure of $\cT_i \cdot \cL \cdot H^0_k$ is isomorphic with a direct sum of irreducible $N$-$N$-subbimodules that are of the form
$$H(\psi_\eta) \ot_N K_i \ot_N H(\psi_\pi) \ot_N \cL' \ot_N H(\psi_{\pi_1}) \ot_N \cdots $$
with $\cL' \in \cF \setminus \cF_0$.
The freeness of $\cF$ and $\Rep(G)$ inside $\FAlg(N)$ implies that two $N$-$N$-bimodules of this last form, for different values of $i$, can never be isomorphic. This proves the claim.

{\bf End of the proof.} By Step 3, $\cK$ contains a non-zero $N$-$N$-subbimodule that belongs to $\cF_0$. Hence, assumption \ref{2B} says that $\bim{M_0}{\cK}{M_0}$ is the trivial $M_0$-$M_0$-bimodule.
\end{proof}

\begin{lemm} \label{lem.step-4}
The von Neumann algebra $P$ is generated by $\{ (\xi^* \ot 1) \psi_\pi(M_0) (\eta \ot 1) \mid \pi \in \irr(G),\ \xi, \eta \in H_\pi \} .$
\end{lemm}

\begin{proof}
Denote by $P_0$ the von Neumann subalgebra of $P$ generated by $\{ (\xi^* \ot 1) \psi_\pi(M_0) (\eta \ot 1) \mid \pi \in \irr(G),\ \xi, \eta \in H_\pi \}$. Taking $\pi = \epsilon$, note that $M_0 \subset P_0$. We have to prove that $P_0 = P$.

By construction (cf.\ Remark \ref{rem.decomposition}), the von Neumann algebra $P$ is densely spanned by
$$
\Bigl \{ (\xi ^*\ot 1)(a_0)_{n+1} (V_{\pi_1})_{1,n+1} \cdots (V_{\pi_n})_{n,n+1}(a_n)_{n+1}(\eta \ot 1)
\; \Big| \; \xi \in \bigotimes_{i = 1}^n H_{\pi_i},\ \eta \in \Bigl( \bigotimes_{i=1}^n H_{\pi_i} \Bigr)^G  \Bigr \}
$$
where $a_0,\ldots,a_n \in M_0$, $\pi_1,\ldots,\pi_n \in \irr(G)$ and where the superscript $^G$ denotes the subspace of $G$-invariant vectors under the tensor product representation.

It therefore suffices to prove by induction on $n$ that for all $\pi_1, \ldots, \pi_n \in \irr(G)$, $\eta \in \Bigl(\bigotimes_{i=1}^n H_{\pi_i} \Bigr)^G$ and $a_0,\ldots,a_n \in M_0$, we have
\begin{equation}\label{eq.W_n}
A_n:=(a_0)_{n+1} (V_{\pi_1})_{1,n+1} \cdots (V_{\pi_n})_{n,n+1} (a_n)_{n+1} (\eta \ot 1) \in H_{\pi_1} \ot \cdots \ot H_{\pi_n} \ot P_0\; .
\end{equation}

The case $n = 1$ being trivial, assume that \eqref{eq.W_n} holds for all $n \leq k-1$. Take $A_k$ as in \eqref{eq.W_n} and re-write $A_k$ in the following way.
$$A_k = (a_0)_{k+1} (\psi_{\pi_1}(a_1))_{1,k+1} (V_{\pi_1})_{1,k+1}(V_{\pi_2})_{2,k+1} \cdots (a_k)_{k+1}(\eta \ot 1)\; . $$
Lemma \ref{lemm.pi ot eta} yields $\mu_1,\ldots,\mu_r \in \irr(G)$, isometries $v_i \in \morph (\mu_i , \pi_1 \ot \pi_2)$ and $X_i \in \B(H_{\mu_i}, H_{\pi_1} \ot H_{\pi_2}) \ot N$ such that
$$(V_{\pi_1})_{13} (V_{\pi_2})_{23} = \sum_{i = 1}^r X_i V_{\mu_i} (v_i^* \ot 1) \; .$$
Put $\xi_i:= (v_i^*)_{12}\, \eta \in (H_{\mu_i} \ot H_{\pi_3} \ot \cdots \ot H_{\pi_n})^G$ and
$$B_i : = (V_{\mu_i})_{1,n} (a_2)_{n}(V_{\pi_3})_{2,n}(V_{\pi_n})_{n-1,n} (a_n)_n ( \xi_i \ot 1 ) \; .$$
By the induction hypothesis, $B_i \in H_{\mu_i} \ot \bigotimes_{i=3}^n H_{\pi_i} \ot P_0$, for all $i$.
Also, $a_0 \in P_0$ and $\psi_{\pi_1}(a_1) \in \B(H_{\pi_1}) \ot P_0$. Since $X_i \in \B(H_{\mu_i}, H_{\pi_1} \ot H_{\pi_2}) \ot P_0$, it follows that
$$A_k = (a_0)_{k+1} (\psi_{\pi_1}(a_1))_{1,k+1} \sum_{i =1}^r  (X_i)_{1,2,k+1}B_i \in  H_{\pi_1} \ot \cdots \ot H_{\pi_k} \ot P_0\; . $$
So, the lemma is proven.
\end{proof}

We can finally prove Theorem \ref{thm.thm}.

\begin{proof}
Since the functor
$$\Rep(G) \droite \Bimod(P) : \pi \mapsto \bim{{P}}{\Mor(H_\pi,\eL^2(M))}{P}$$
is a fully faithful tensor functor, it suffices to prove that for every irreducible finite index $P$-$P$-bimodule $\bim{P}{\cH}{P}$, there exists an $\eta \in \irr(G)$ such that $\bim{P}{\cH}{P} \cong \bim{\psi_\eta(P)}{(H_\eta \ot \eL^2(P))}{P}$. So, let $\bim{P}{\cH}{P}$ be an irreducible finite index $P$-$P$-bimodule.

Decompose $\cH$ into a direct sum $\cH = \bigoplus_{i=1}^k \cH^i$ of irreducible $M_0$-$P$-subbimodules. By Lemma \ref{lem.step-2}, we can take $\eta_1,\ldots,\eta_k \in \irr(G)$ and non-zero irreducible $M_0$-$\psi_{\eta_i}(M_0)$-subbimodules $\cK^i \subset H_{\eta_i}^* \otimes \cH^i$ such that $\bim{M_0}{\cK^i}{\psi_{\eta_i}(M_0)}$ has finite index. Viewing $\cK^i$ as an $M_0$-$\psi_{\eta_i}(M_0)$-subbimodule of the finite index bimodule $\bim{P}{(H_{\eta_i}^* \ot \cH)}{\psi_{\eta_i}(P)}$, Lemma \ref{lem.step-3} says that $\bim{M_0}{\cK^i}{\psi_{\eta_i}(M_0)} \cong \bim{M_0}{\eL^2(M_0)}{M_0}$.

Take a finite index inclusion $\psi : P \recht p P^n p$ such that $\bim{P}{\cH}{P} \cong \bim{\psi(P)}{p(\C^n \ot \eL^2(P))}{P}$. Denote $A = \psi(M_0)' \cap p P^n p$. Then, $A$ is finite dimensional and $\cH^i$ corresponds to $p_i(\C^n \ot \eL^2(P))$, where $p_1,\ldots,p_k$ are minimal projections in $A$ summing to $1$. In the previous paragraph, it was shown that $\bim{\psi(M_0)p_i}{p_i(\C^n H_{\eta_i}^* \ot \eL^2(P))}{\psi_{\eta_i}(M_0)}$ contains the trivial $M_0$-$\psi_{\eta_i}(M_0)$-bimodule. So, we can take non-zero vectors $v_i \in
p_i(\C^n H_{\eta_i}^* \ot \eL^2(P))$ satisfying $\psi(a) v_i = v_i \psi_{\eta_i}(a)$ for all $a \in M_0$. As an element of $p_i(\M_n(\C) \ot \rL^1(P))p_i$, the operator $v_i v_i^*$ commutes with $\psi(M_0)$. By minimality of $p_i$, it follows that $v_iv_i^*$ is a multiple of $p_i$ and in particular, $v_i \in p_i(\C^n H_{\eta_i}^* \ot P)$. So, we may assume that $v_iv_i^* = p_i$. On the other hand, $v_i^* v_i \in \B(H_{\eta_i}) \ot P \cap \psi_{\eta_i}(M_0)' = \C 1$, implying that $v_i^* v_i = 1$.

Define $\eta = \bigoplus_{i=1}^k \eta_i$ and put $\psi_\eta : P \recht \B(H_\eta) \ot P : \psi_\eta(a) = \bigoplus_{i=1}^k \psi_{\eta_i}(a)$. Note that $\psi_\eta(a) = V_\eta (1 \ot a) V_\eta^*$, where $V_\eta = \bigoplus_{i=1}^k V_{\eta_i}$. We have shown the existence of $u \in p (\C^n H_\eta^* \ot P)$ satisfying $uu^* = p$, $u^* u = 1$ and $u^* \psi(a) u = \psi_\eta(a)$ for all $a \in M_0$. We may assume from now on that $\bim{P}{\cH}{P} \cong \bim{\psi(P)}{(H_\eta \ot \eL^2(P))}{P}$ where $\psi : P \recht \B(H_\eta) \ot P$ is a finite index inclusion satisfying $\psi(a) = \psi_\eta(a)$ for all $a \in M_0$. It remains to prove that $\psi(a) = \psi_\eta(a)$ for all $a \in P$.

By Lemma \ref{lem.step-4}, it suffices to prove that $(\id \ot \psi)\psi_\pi(a) = (\id \ot \psi_\eta)\psi_\pi(a)$ for all $a \in M_0$ and all $\pi \in \irr(G)$. Fix $\pi \in \irr(G)$. Applying the reasoning in the previous paragraphs to the $P$-$P$-bimodule $\bim{\psi_\pi(P)}{(H_\pi \ot \cH)}{P}$, we find a finite dimensional unitary representation $\gamma$ of $G$ and a unitary $X \in \B(H_\pi \ot H_\eta,H_\gamma) \ot P$ satisfying
\begin{equation}\label{eq.tussenstap}
(\id \ot \psi)\psi_\pi(a) = X \psi_\gamma(a) X^*
\end{equation}
for all $a \in M_0$. Because $\psi_\pi(N) \subset \B(H_\pi) \ot N$ and $\psi(a) = \psi_\eta(a)$ for all $a \in N$, it follows in particular that $(\id \ot \psi_\eta)\psi_\pi(a) = X \psi_\gamma(a) X^*$ for all $a \in N$. Hence, the unitary
$$Z := (V_\eta)^*_{23} (V_{\pi})^*_{13} X V_\gamma$$
satisfies $Z (1 \ot a) = (1 \ot 1 \ot a) Z$ for all $a \in N$. So, $Z = U \ot 1$, where $U : H_\gamma \recht H_\pi \ot H_\eta$ intertwines the representations $\gamma$ and $\pi \ot \eta$. It follows that $X \psi_\gamma(a) X^* = (\id \ot \psi_\eta)\psi_\pi(a)$ for all $a \in M_0$. Combining with \eqref{eq.tussenstap}, we conclude that $(\id \ot \psi)\psi_\pi(a) = (\id \ot \psi_\eta)\psi_\pi(a)$ for all $a \in M_0$.
\end{proof}

As a consequence of Theorem \ref{thm.thm}, we now prove Theorem \ref{thm.main} stated in the introduction.

\begin{proof}[Proof of Theorem \ref{thm.main}]
Denote by $M_1$ the hyperfinite II$_1$ factor and take a minimal action $G \actson M_1$. Put $N := M_1^G$.

Define the group $\Gamma$, its subgroup $\Lambda < \Gamma$ and the scalar $2$-cocycle $\Omega \in \rZ^2(\Gamma,S^1)$ as in Example \ref{exam.our-example}. Write $R := \rL_\Omega(\Lambda)$ and $M_0 := \rL_\Omega(\Gamma)$. Denote by $\cF$ the fusion subalgebra of $\FAlg(R)$ generated by all finite index $R$-$R$-bimodules appearing as an $R$-$R$-subbimodule of a finite index $M_0$-$M_0$-bimodule. By Remark \ref{rem.countable}, $\cF$ is countable.

Note that both $N$ and $R$ are isomorphic with the hyperfinite II$_1$ factor. Whenever $\al : N \recht R$ is an isomorphism, we can view $\al^{-1} \cF \al$ as a fusion subalgebra of $\FAlg(N)$. By Theorem \ref{Gdelta}, we can choose $\al$ such that $\al^{-1} \cF \al$ is free w.r.t.\ the image of $\Rep(G)$ inside $\FAlg(N)$. Identifying $N$ and $R$ through $\al$, it follows from Example \ref{exam.our-example} that all assumptions for Theorem \ref{thm.thm} are satisfied.

So, we can take $M := M_0 *_N M_1$, extend $G \actson M_1$ to a minimal action $G \actson M$ by acting trivially on $M_0$ and conclude from Theorem \ref{thm.thm} that the natural tensor functor $\Rep(G) \recht \Bimod(M^G)$ is an equivalence of categories.
\end{proof}

\section{Proof of Theorem \ref{thm.subfactors}}

Fix a second countable compact group $G$ and a minimal action $G \overset{\sigma}{\actson} M$ on a II$_1$ factor $M$. Let $G \overset{\alpha}{\actson} A$ be an action of $G$ on the finite dimensional von Neumann algebra $A$. Denote by $\al \ot \si$ the diagonal action of $G$ on $A \ot M$, given by $(\al \ot \si)_g = \al_g \ot \si_g$ for all $g \in G$. Denote by $N^\beta$ the fixed point algebra of an action $\beta$ on a von Neumann algebra $N$.

\begin{lemm}\label{lemma.intermediate}
The fixed point algebra $(A \ot M)^{\alpha \ot \sigma}$ is a factor iff $\cZ(A)^\al = \C 1$. The inclusion $1 \ot M^\si \subset (A \ot M)^{\al \ot \sigma}$ is irreducible iff $A^\al = \C 1$.

Every intermediate von Neumann algebra $1 \ot M^\si \subset N \subset (A \ot M)^{\alpha \ot \sigma}$ is of the form $(D \ot M)^{\al \ot \si}$ for a uniquely determined globally $G$-invariant $*$-subalgebra $D \subset A$.
\end{lemm}

\begin{proof}
Denote $P := M^\si$. By minimality, the relative commutant of $1 \ot P$ inside $(A \ot M)^{\alpha \ot \sigma}$ equals $A^\al \ot 1$. So, the inclusion $1 \ot P \subset (A \ot M)^{\alpha \ot \sigma}$ is irreducible iff $A^\al = \C 1$. Also, $\cZ\bigl((A \ot M)^{\al \ot \si}\bigr) \subset A^\al \ot 1$.

We claim that
\begin{equation}\label{eq.nuttig}
A = \lspan \{ (\id \ot \om)(a) \mid a \in (A \ot M)^{\al \ot \si} \; , \; \om \in M_* \} \; .
\end{equation}
In order to prove this claim, let $\pi \in \irr(G)$, $X \in H_\pi^* \ot A$ and $(\id \ot \al_g)(X) = X(\pi(g) \ot 1)$ for all $g \in G$. Note that $A$ is the linear span of all possible $X (H_\pi \ot 1)$. On the other hand, $X_{12} (V_\pi)^*_{13}$ belongs to $H_\pi^* \ot (A \ot M)^{\al \ot \si}$, implying that $X(H_\pi \ot 1)$ is included in the expression at the right-hand side of \eqref{eq.nuttig}. This proves \eqref{eq.nuttig}.

A combination of the claim and the first paragraph of the proof implies that $\cZ\bigl((A \ot M)^{\al \ot \si}\bigr) = \cZ(A)^\al \ot 1$. Hence, $(A \ot M)^{\al \ot \si}$ is a factor iff $\cZ(A)^\al = \C 1$.

Let $1 \ot M^\si \subset N \subset (A \ot M)^{\alpha \ot \sigma}$ be an intermediate von Neumann algebra. Choose a $G$-invariant trace on $A$, so that we can view $A$ as a finite-dimensional Hilbert space. In this way, the action $\al$ can be seen as a unitary representation $\pi_A : G \recht \cU(A)$. Choose a unitary $W \in \B(A) \ot M$ satisfying $(\id \ot \si_g)(W) = W (\pi_A(g) \ot 1)$ for all $g \in G$. Define the finite index inclusion $\gamma : P \recht \B(A) \ot P : \gamma(a) = W(1 \ot a)W^*$. The map $a \mapsto Wa$ defines a $P$-$P$-bimodular unitary
$$\Theta : \bim{{(1 \ot P)}}{\rL^2\bigl((A \ot M)^{\al \ot \si}\bigr)}{(1 \ot P)} \recht \bim{{\gamma(P)}}{\bigl(A \ot \rL^2(P)\bigr)}{P} \; .$$
It follows that $\Theta(N) = q (A \ot \rL^2(P))$, where $q$ is a projection in $$\B(A) \ot P \cap \gamma(P)' = W(\B(A)^{\Ad \pi_A} \ot 1)W^* \; .$$ Write $q = W (p \ot 1) W^*$ and define the vector subspace $D \subset A$ as the image of the projection $p$. Since $p$ commutes with $\pi_A(G)$, it follows that $D$ is globally $\al$-invariant. We have shown that
$$N = (A \ot M)^{\al \ot \si} \cap (D \ot M) = (D \ot M)^{\al \ot \si} \; .$$
It remains to prove that $D$ is a $*$-algebra.

Since $D$ is globally $\al$-invariant, it follows that $D$ is linearly spanned by elements of the form $X(H_\pi \ot 1)$, where $\pi \in \irr(G)$, $X \in H_\pi^* \ot D$ and $(\id \ot \al_g)(X) = X(\pi(g) \ot 1)$ for all $g \in G$. As in \eqref{eq.nuttig}, it follows that $D$ is linearly spanned by $(\id \ot \om)(a)$, $\om \in M_*$ and $a \in N$. Hence, $D = D^*$.

Further, let $\pi,\eta \in \irr(G)$, $X \in H_\pi^* \ot D$, $Y \in H_\eta^* \ot D$ and $(\id \ot \al_g)(X) = X(\pi(g) \ot 1)$, $(\id \ot \al_g)(Y) = Y (\eta(g) \ot 1)$ for all $g \in G$. To conclude the proof of the lemma, it suffices to show that $X_{13} Y_{23} \in (H_\pi \ot H_\eta)^* \ot D$. But, we know that $X_{12} (V_{\pi})^*_{13} \in H_\pi^* \ot N$ and $Y_{12} (V_\eta)^*_{13} \in H_\eta^* \ot N$. Since $N$ is an algebra, it follows that
$$X_{13} \; (V_\pi)^*_{14} \; Y_{23} \; (V_\eta)^*_{24} \in (H_\pi \ot H_\eta)^* \ot N \subset (H_\pi \ot H_\eta)^* \ot D \ot M \; .$$
The two factors in the middle commute and the conclusion follows.
\end{proof}

We are now ready to prove Theorem \ref{thm.subfactors}.

\begin{proof}[Proof of Theorem \ref{thm.subfactors}]
Take $G \actson M$ as in the formulation of the theorem and put $P := M^\si$. Let $P_0 \subset P$ be a finite index subfactor. We first prove that $P_0$ is unitarily conjugate in $P$ to a subfactor of the form $P(\al)$ for some action $G \overset{\al}{\actson} A$ of $G$ on a finite dimensional von Neumann algebra $A$ satisfying $\cZ(A)^\al = \C 1$.

Let $P_0 \subset P \subset P_1$ be the basic construction. Then, $\bim{P}{\rL^2(P_1)}{P}$ is a finite index $P$-$P$-bimodule. By assumption, we find a finite dimensional unitary representation $\pi : G \recht \cU(n)$ and a unitary $V \in \M_n(\C) \ot M$ satisfying $(\id \ot \si_g)(V) = V (\pi(g) \ot 1)$ for all $g \in G$, such that
$$\bim{P}{\rL^2(P_1)}{P} \cong \bim{{\gamma(P)}}{\bigl(\C^n \ot \rL^2(P)\bigr)}{P} \quad\text{where}\quad \gamma(a) = V(1 \ot a) V^* \;\;\text{for all}\; a \in P \; .$$
The left $P_1$-action on $\rL^2(P_1)$ commutes with the right $P$-action and hence, we can extend $\gamma$ to an inclusion $\gamma : P_1 \recht \M_n(\C) \ot P$. Denote $N = V^* \gamma(P_1) V$ and write $\al(g) = \Ad(\pi(g))$. It follows that $1 \ot P \subset N \subset (\M_n(\C) \ot M)^{\al \ot \si}$. Applying Lemma \ref{lemma.intermediate}, we find a finite dimensional von Neumann algebra $A$, an action $G \overset{\al}{\actson} A$ satisfying $\cZ(A)^\al = \C 1$ and a $*$-isomorphism $\theta : P_1 \recht (A \ot M)^{\al \ot \si}$ satisfying $\theta(a) = 1 \ot a$ for all $a \in P$. By uniqueness of the tunnel construction, it follows that $P_0$ and $P(\al)$ are unitarily conjugate.

Finally, suppose that $G \overset{\alpha}{\actson} A$ and $G \overset{\beta}{\actson} B$ satisfy $\cZ(A)^\al = \C1$ and $\cZ(B)^\beta = 1$ and suppose that the subfactors $P(\al)$ and $P(\beta)$ are unitarily conjugate in $P$. It remains to construct a $*$-isomorphism $\pi : A \recht B$ satisfying $\beta_g \circ \pi = \pi \circ \al_g$ for all $g \in G$. By assumption, we find a $*$-isomorphism $\theta : (A \ot M)^{\al \ot \si} \recht (B \ot M)^{\be \ot \si}$ satisfying $\theta(1 \ot a) = 1 \ot a$ for all $a \in P$. Repeating the argument given in the proof of Lemma \ref{lemma.intermediate}, we find a bijective linear map $\gamma : A \recht B$ such that $\gamma \circ \al_g = \beta_g \circ \gamma$ for all $g \in G$ and $\theta(b) = (\gamma \ot \id)(b)$ for all $b \in (A \ot M)^{\al \ot \si}$. By \eqref{eq.nuttig}, it follows that $\gamma$ is a $*$-isomorphism.
\end{proof}

\end{document}